\pdfoutput=1
 \documentclass[11pt]{amsart}
\usepackage[T1]{fontenc}
\usepackage[utf8]{inputenc}
\usepackage{prettyref}
\usepackage{mathtools}
\usepackage{dsfont}
\usepackage{amstext}
\usepackage{amsthm}
\usepackage{amssymb}
\usepackage{geometry}
\usepackage{microtype}
\usepackage[bookmarks=false,
 breaklinks=false,pdfborder={0 0 1},backref=false,colorlinks=false]
 {hyperref}
\hypersetup{pdftitle={Geometric mean quantization via adaptive approximation},
 pdfauthor={Marc Kesseböhmer and Aljoscha Niemann},
 pdfsubject={Mass-threshold quantization, entropy, and local dimensions},
 pdfkeywords={geometric-mean quantization, entropy dimension, mass-threshold stopping partitions, adaptive approximation algorithm, local dimension, Lq-spectrum, Bernoulli mixtures, Dirichlet series}}

\makeatletter
\numberwithin{equation}{section}
\numberwithin{figure}{section}

\newtheorem{theorem}{Theorem}[section]
\newtheorem{proposition}[theorem]{Proposition}\newtheorem{corollary}[theorem]{Corollary}\newtheorem{lemma}[theorem]{Lemma}\theoremstyle{definition}
\newtheorem{example}[theorem]{Example}\theoremstyle{remark}
\newtheorem{remark}[theorem]{Remark}\numberwithin{equation}{section}

\newcommand{\D}{\mathcal D}
\newcommand{\R}{\mathbb R}

\renewcommand{\log}{\operatorname{log}_{2}}
\renewcommand{\d}{\;\mathrm{d}}
\newcommand{\e}{\mathrm{e}}
\makeatother

\theoremstyle{plain}

\providecommand{\propositionname}{Proposition}
\providecommand{\theoremname}{Theorem}

\begin{document}
\title[Geometric mean quantization via adaptive approximation]{ Geometric mean quantization via adaptive approximation}
\author{Marc Kesseböhmer}
\address{Institute for Dynamical Systems, FB 3 -- Mathematics and Computer
Science, University of Bremen, Bibliothekstr. 5, 28359 Bremen, Germany}
\email{mhk@uni-bremen.de}
\author{Aljoscha Niemann}
\address{Institute for Dynamical Systems, FB 3 -- Mathematics and Computer
Science, University of Bremen, Bibliothekstr. 5, 28359 Bremen, Germany}
\email{niemann1@uni-bremen.de}
\date{September 13, 2026}
\begin{abstract}
Let $\nu$ be a compactly supported Borel probability measure on $\R^{d}$
with $\nu(B(x,r))\leq Cr^{a}$ for some $a>0$. Refine a dyadic cube
exactly when its mass is at least $t$, and let $\mathcal{L}_{\nu}(t)$
be the mean depth at which this refinement stops. We show that the
lower and upper geometric-mean quantization dimensions of $\nu$ are
the lower and upper limits of $\log(1/t)/\mathcal{L}_{\nu}(t)$. The
dimension exists precisely when $(q-1)\sum_{Q}\nu(Q)^{q}$, summed
over all dyadic cubes, converges as $q\downarrow1$, and it is then
determined by this limit. The mass-threshold formula yields harmonic
integral bounds in terms of the local dimensions and encloses entropy
and quantization dimensions in a common spectral interval. Convergence
in law of the local information rates is equivalent to convergence
of the rescaled spectra in a window of width $1/k$ around $q=1$;
the two dimensions are then the arithmetic and the harmonic mean of
the limit law, and we quantify their difference by sharp bounds and
variance identities. Without any convergence assumption, vanishing
threshold variance still forces equality of the corresponding lower
and upper dimensions. Bernoulli mixtures realise every local-dimension
law with compact support in $(0,1]$, and a regime-switching example
separates convergence in law from almost-everywhere convergence.
\end{abstract}

\maketitle

\section{Introduction}

Throughout, the natural logarithm is denoted by $\ln$ and the logarithm
to base 2 by $\log$. In quantization theory, a Borel probability
distribution is approximated by a set of at most $N$ points that
minimise the corresponding geometric-mean error, denoted by $\mathfrak{e}_{N,0}(\nu)$.
Its lower and upper dimensions (scaling limit), $\underline{D}_{0}(\nu)$
and $\overline{D}_{0}(\nu)$, need not equal the fixed-scale entropy
dimensions $\underline{h}(\nu)$ and $\overline{h}(\nu)$. All definitions
and dimension conventions are collected in Section~\ref{sec:notation}
and are not repeated here. The distinction between the two notions
is already visible in Graf--Luschgy's harmonic mixture rule \cite[Example~4.1]{GL04}
and Zhu's lower/upper mixture bounds \cite[Lemma~2.2]{Zhu12}.

\subsection{Main result}

\subsubsection*{\textbf{A mass-threshold formula}}

We work on the dyadic root cube $[0,1]^{d}$ with the maximum norm,
assume that the support of $\nu$ lies in the interior of the unit
cube, and impose the uniform polynomial mass bound \eqref{eq:uniform-dyadic-mass},
$\nu(Q)\leq C2^{-ak}$ for all dyadic cubes $Q$ of generation $k$,
which is equivalent to $\dim_{\infty}(\nu)>0$. For a threshold $0<t<1$
let $P_{\nu}(t)$ be the family of dyadic cubes at which the mass
first drops below $t$ along a dyadic path, let $\tau_{t}(x)$ be
the generation of the cube of $P_{\nu}(t)$ containing $x$, i.e.
\begin{equation}
\tau_{t}(x)\coloneqq\min\{k\geq1:\nu(Q_{k}(x))<t\}\label{eq:stopping-time}
\end{equation}
 where $Q_{k}(x)$ is the unique dyadic cube of level $k$, containg
$x$, and let $\mathcal{L}_{\nu}(t)=\int\tau_{t}\d\nu$ be the $\nu$-mean
depth of the partition; see \eqref{eq:mass-stopping-partition} and
\eqref{eq:stopping-cardinality-def}. Theorem~\ref{thm:adaptive-entropy}
establishes
\[
\underline{D}_{0}(\nu)=\liminf_{t\downarrow0}\frac{\log(1/t)}{\mathcal{L}_{\nu}(t)},\qquad\overline{D}_{0}(\nu)=\limsup_{t\downarrow0}\frac{\log(1/t)}{\mathcal{L}_{\nu}(t)}.
\]
The theorem compares an information threshold with the mean spatial
depth needed to reach it. New is this general dyadic formulation,
covering both limits, without a dynamical model, separation, or lower
child-to-parent mass bounds.

\subsubsection*{\textbf{The adaptive algorithm and its critical exponent}}

The family $P_{\nu}(t)$ is the minimal $t$-good partition $G_{1/t}$
of \cite[(1.1)]{KN25} for the monotone set function $\mathfrak{J}=\nu$,
the output of the adaptive approximation algorithm of \cite[Section~1.3]{KN25};
Remark~\ref{rem:adaptive-algorithm} records the dictionary between
the two frameworks. In \cite{KN25} such algorithms are governed by
the critical exponent
\[
\kappa(\mathfrak{J})\coloneqq\inf\Bigl\{ q>0:\sum_{Q\in\mathcal{D}}\mathfrak{J}(Q)^{q}<\infty\Bigr\},
\]
the abscissa of convergence of the $\mathfrak{J}$-series: it is the
zero of the partition function \cite[Lemma~2.4]{KN25}, it gives the
growth exponent of the cardinality $\#G_{x}$ of the output, and for
$\mathfrak{J}_{\nu,r}(Q)=\sup_{Q'\subseteq Q}\nu(Q')\ell(Q')^{r}$,
$r\neq0$, it gives the quantization dimensions of order $r$ \cite{KNZ,KNneg}.

Order zero is the case $\mathfrak{J}=\nu$, with critical exponent
$\kappa_{\nu}=1$: the dimension is carried not by the cardinality
of the adaptive partition but by its mean depth $\mathcal{L}_{\nu}(t)$,
the expected stopping time of the algorithm, and equivalently by the
rate at which the \emph{(Dirichlet) series} $Z_{\nu}(q)=\sum_{Q\in\mathcal{D}}\nu(Q)^{q}$
blows up at $\kappa_{\nu}$. Proposition~\ref{prop:mean-depth-dual}
identifies $\mathcal{L}_{\nu}(t)$ with the $\nu$-weighted number
of dyadic cubes of mass at least $t$ and shows that the threshold
partitions maximise the mean depth among dyadic partitions of given
cardinality. In the spirit of \cite{KNZ,KN25}, where the quantization
dimensions of order $r\neq0$ are the growth exponents of the number
of cubes the algorithm produces for $\mathfrak{J}_{\nu,r}$, the order-zero
quantization problem is thereby reduced to a pure counting problem
for the same algorithm, each subdivided cube counted with its mass.
Corollary~\ref{cor:tauberian} rests on the transform identity $Z_{\nu}(1+\lambda)=\int_{[0,\infty)}2^{-\lambda u}\d F_{\nu}(u)$
applied to the increasing counting function $F_{\nu}:u\mapsto\mathcal{L}_{\nu}(2^{-u})$.
Its Abelian half is elementary: $(q-1)\ln2\,Z_{\nu}(q)$ is a Gamma-kernel
average of $\mathcal{L}_{\nu}(t)/\log(1/t)$, whence
\[
\frac{1}{\overline{D}_{0}(\nu)}\leq\liminf_{q\downarrow1}(q-1)\ln2\,Z_{\nu}(q)\leq\limsup_{q\downarrow1}(q-1)\ln2\,Z_{\nu}(q)\leq\frac{1}{\underline{D}_{0}(\nu)}.
\]
Its Tauberian half, Karamata's theorem for the monotone function $F_{\nu}$,
turns the existence of the limit into the existence of the dimension:
$D_{0}(\nu)$ exists if and only if $(q-1)Z_{\nu}(q)$ converges as
$q\downarrow1$, and then
\[
D_{0}(\nu)=\Bigl(\lim_{q\downarrow1}(q-1)\ln2\sum_{Q\in\mathcal{D}}\nu(Q)^{q}\Bigr)^{-1}.
\]
The dimension at order zero is thus the reciprocal of the normalised
real-axis residue of $Z_{\nu}$ at its critical exponent (Remark~\ref{rem:residue}).

Section~\ref{sec:adaptive-spectrum} adds the entropy of the adaptive
partition, $\mathcal{H}_{\nu}(t)=\log(1/t)+O(\log\log(1/t))$ of \eqref{eq:stopping-entropy-def}:
the two dimensions are the lower and upper limits of the entropy per
unit depth $\mathcal{H}_{\nu}(t)/\mathcal{L}_{\nu}(t)$, an exact
weighted harmonic mean of the cube-wise dimensions $\log\nu(Q)/\log\ell(Q)$
(see Lemma~\ref{prop:harmonic-stopping}), the finite-scale form
of the harmonic formula below. Thus the quantization dimension is
the entropy per level of refinement of the adaptive tree, in the same
way as the entropy dimensions $\underline{h}(\nu),\overline{h}(\nu)$
are the lower and upper limits of the entropy per level $H_{k}(\nu)/k$
of the uniform tree; Example~\ref{ex:two-spectra} shows that the
normalised $L^{q}$-sums of the adaptive partition are not the ordinary
$L^{q}$-spectrum.

\subsubsection*{\textbf{A common spectral enclosure}}

Let $\overline{\beta}$ be the upper dyadic $L^{q}$-spectrum and
$[\sigma_{-},\sigma_{+}]=-\partial\overline{\beta}(1)$ its reflected
subdifferential at $1$, see \eqref{eq:reflected-subdifferential}.
Theorem~\ref{thm:spectral-enclosure} collects the classical bounds
of Heurteaux \cite[Theorems~3.1--3.2]{Heu07} and Zhu \cite[Theorem~2.1]{Zhu12}
in the form
\[
\begin{aligned}\sigma_{-} & \leq\dim_{*}(\nu)\leq\underline{h}(\nu)\leq\overline{h}(\nu)\leq\mathrm{Dim}^{*}(\nu)\leq\sigma_{+},\\
\sigma_{-} & \leq\dim_{*}(\nu)\leq\underline{D}_{0}(\nu)\leq\overline{D}_{0}(\nu)\leq\mathrm{Dim}^{*}(\nu)\leq\sigma_{+}.
\end{aligned}
\]
 Lemma~\ref{lem:critical-zero} locates the critical parameters $q_{r}$
of \cite{KNZ,KNneg} on $\overline{\beta}$ also for $-\dim_{\infty}(\nu)<r<0$,
and with \cite[Theorem~1.5(i)--(ii)]{KNneg} this gives $\lim_{r\uparrow0}\underline{D}_{r}(\nu)=\sigma_{-}$
and $\lim_{r\downarrow0}\overline{D}_{r}(\nu)=\sigma_{+}$ as conclusions,
not as regularity assumptions. Differentiability of $\overline{\beta}$
at $1$ collapses the enclosure to a point, the order-zero statement
of \cite[Theorem~1.5(iii)]{KNneg}; the converse fails even for an
almost everywhere constant local dimension (Example~\ref{ex:law-not-ae}).

\subsubsection*{\textbf{The rescaled spectra at $q=1$ and the two means}}

Write $X_{k}(x)=-\log\nu(Q_{k}(x))/k$ for the local information rates
and
\[
\Psi_{k}(\theta)\coloneqq k\ln2\,\beta_{k}\Bigl(1+\frac{\theta}{k\ln2}\Bigr),
\]
the finite-generation spectrum rescaled in a window of width $1/k$
around $q=1$. Proposition~\ref{prop:spectral-window} shows that
the laws of $X_{k}$ under $\nu$ converge weakly to a probability
measure $\varrho$ if and only if $\Psi_{k}(\theta)$ converges for
every $\theta\in\R$; then $\varrho$ is carried by the reflected
subdifferential $J_{\nu}=[\sigma_{-},\sigma_{+}]$, $\Psi=\lim_{k}\Psi_{k}$
is the $\log$-Laplace transform of $-\varrho$, and both dimensions
exist as the arithmetic and the harmonic mean of $\varrho$:
\[
h(\nu)=\int y\d\varrho(y)=-\Psi'(0),\qquad D_{0}(\nu)=\Bigl(\int y^{-1}\d\varrho(y)\Bigr)^{-1}=\Bigl(\int^{\infty}_{0}e^{\Psi(\theta)}\d\theta\Bigr)^{-1}.
\]
The window carries the whole law of the local dimension: the normalised
derivatives $\beta^{(n)}_{k}(1)/(k\ln2)^{n-1}$ are the cumulants
of $-X_{k}$, the negative of the first converging to $h(\nu)$ and
the second controlling the defect $h(\nu)-D_{0}(\nu)$. The hypothesis
is strictly weaker than almost-everywhere convergence: Example~\ref{ex:law-not-ae}
gives a measure with $\varrho=\delta_{s}$, hence $h(\nu)=D_{0}(\nu)=s$,
whose local dimension exists almost nowhere.

Almost-everywhere convergence is the case $\varrho=s_{*}\nu$: if
the dyadic local dimension $s(x)=\lim_{k}X_{k}(x)$ exists almost
everywhere, the display becomes
\[
h(\nu)=\int s\d\nu,\qquad D_{0}(\nu)=\Bigl(\int s^{-1}\d\nu\Bigr)^{-1},
\]
the classical arithmetic and harmonic means of the local dimension
(Remark~\ref{cor:local-harmonic}), which is also obtained directly
from Theorem~\ref{thm:adaptive-entropy}.

Theorem~\ref{thm:quantitative-gap} sharpens the comparison at the
level of $\varrho$. With $m_{\varrho}=\min\operatorname{supp}\varrho$
and $M_{\varrho}=\max\operatorname{supp}\varrho$, which satisfy $\sigma_{-}\leq m_{\varrho}\leq M_{\varrho}\leq\sigma_{+}$,
\[
\frac{m_{\varrho}M_{\varrho}}{m_{\varrho}+M_{\varrho}-h(\nu)}\leq D_{0}(\nu)\leq h(\nu),\qquad h(\nu)-D_{0}(\nu)\leq(\sqrt{M_{\varrho}}-\sqrt{m_{\varrho}})^{2},
\]
the defect is $h(\nu)-D_{0}(\nu)=\int(y-D_{0}(\nu))^{2}y^{-1}\d\varrho(y)$,
and the normalised curvature $\beta_{k}''(1)/(k\ln2)$ converges to
$\operatorname{Var}(\varrho)$. Under almost-everywhere convergence
$m_{\varrho}=\dim_{*}(\nu)$ and $M_{\varrho}=\mathrm{Dim}^{*}(\nu)$
are the lower Hausdorff and the upper packing dimension of $\nu$,
and the defect is enclosed in two tiers,
\[
h(\nu)-D_{0}(\nu)\leq\bigl(\sqrt{\mathrm{Dim}^{*}(\nu)}-\sqrt{\dim_{*}(\nu)}\bigr)^{2}\leq(\sqrt{\sigma_{+}}-\sqrt{\sigma_{-}})^{2},
\]
classical dimensions inside and the spectrum outside, as in the enclosure
above; the inner bound is sharp, the outer one is computable from
the spectrum alone and can be strictly weaker (Example~\ref{ex:law-not-ae}).
Bernoulli mixtures with a random parameter (Section~\ref{sec:bernoulli-mixtures})
realise every prescribed $\varrho$ with compact support in $(0,1]$,
hence every admissible pair $(h(\nu),D_{0}(\nu))$ in the displayed
region, all with full support.

\subsubsection*{\textbf{Integral bounds without convergence}}

Without any convergence assumption, with $\underline{s}$ and $\overline{s}$
the lower and upper dyadic local dimensions, Proposition~\ref{prop:local-liminf-bounds}
gives
\[
\left(\int1/\underline{s}\d\nu\right)^{-1}\leq\underline{D}_{0}(\nu)\leq\overline{D}_{0}(\nu)\leq\left(\int1/\overline{s}\d\nu\right)^{-1},\qquad\int\underline{s}\d\nu\leq\underline{h}(\nu)\leq\overline{h}(\nu)\leq\int\overline{s}\d\nu.
\]
The harmonic bounds follow from Theorem~\ref{thm:adaptive-entropy}
by Fatou's lemma and its reverse, since the dimension estimate $\varsigma_{t}(x)\coloneqq\log(1/t)/\tau_{t}(x)$
of the algorithm has lower and upper limits $\underline{s}(x)$ and
$\overline{s}(x)$ and $1/\varsigma_{t}$ is bounded. Replacing the
integral means by essential extrema gives Zhu's extremal chain $\dim_{*}(\nu)\leq\underline{D}_{0}(\nu)\leq\overline{D}_{0}(\nu)\leq\mathrm{Dim}^{*}(\nu)$
\cite[Theorem~2.1]{Zhu12} (Remark~\ref{rem:heurteaux-dimensions});
conversely, Remark~\ref{prop:recovery-zhu} assembles the integral
bounds from Zhu's chain and the harmonic mixture rule of Graf--Luschgy
\cite[Example~4.1]{GL04}. The two formulations are equivalent, but
the mass-threshold formula yields the integral form, which appears
to be new, directly and for both limits at once; the arithmetic bounds
are known \cite{FLR02,ST12}.

\subsubsection*{\textbf{The upper variance of the local dimension without convergence}}

Theorem~\ref{thm:adaptive-entropy} concerns the harmonic mean of
the dimension estimate $\varsigma_{t}$: $\log(1/t)/\mathcal{L}_{\nu}(t)=(\int\varsigma^{-1}_{t}\d\nu)^{-1}$
exactly. Its arithmetic mean $\mathsf{h}_{\nu}(t)\coloneqq\int\varsigma_{t}\d\nu$
and its variance $\mathsf{V}_{\nu}(t)\coloneqq\operatorname{Var}_{\nu}(\varsigma_{t})$
define, without any convergence hypothesis, the upper threshold variance
$\overline{\mathsf{V}}(\nu)\coloneqq\limsup_{t\downarrow0}\mathsf{V}_{\nu}(t)$
of the local dimension. Thus $\mathsf{h}_{\nu}(t)$ is the threshold-indexed
counterpart of the entropy per level $h_{k}(\nu)=H_{k}(\nu)/k=\int X_{k}\d\nu$,
whose lower and upper limits are the entropy dimensions $\underline{h}(\nu),\overline{h}(\nu)$;
it is not the entropy $\mathcal{H}_{\nu}(t)$ of the stopping partition,
which is of order $\log(1/t)$. Both indexings obey the bounds $\int\underline{s}\d\nu\leq\cdot\leq\int\overline{s}\d\nu$
of Proposition~\ref{prop:local-liminf-bounds}, and they have the
same limits under vanishing threshold variance and under convergence
in law (Proposition~\ref{prop:threshold-variance}(3),(4)). Proposition~\ref{prop:threshold-variance}
gives at every threshold
\[
\mathsf{h}_{\nu}(t)-\frac{\log(1/t)}{\mathcal{L}_{\nu}(t)}=\int\frac{(\varsigma_{t}-\log(1/t)/\mathcal{L}_{\nu}(t))^{2}}{\varsigma_{t}}\d\nu,
\]
with two-sided bounds by $\mathsf{V}_{\nu}(t)$, so that
\begin{align*}
\overline{\mathsf{V}}(\nu)=0 & \iff\lim_{t\downarrow0}\mathsf{h}_{\nu}(t)-\frac{\log(1/t)}{\mathcal{L}_{\nu}(t)}=0\\
 & \implies\underline{h}(\nu)=\liminf_{t\downarrow0}\mathsf{h}_{\nu}(t)=\underline{D}_{0}(\nu),\ \overline{h}(\nu)=\limsup_{t\downarrow0}\mathsf{h}_{\nu}(t)=\overline{D}_{0}(\nu),\ \beta_{k}''(1)=o(k).
\end{align*}
Under the hypothesis of Proposition~\ref{prop:spectral-window} $\left(\varsigma_{t}\right)$
converges in law to $\varrho$ as well, $\overline{\mathsf{V}}(\nu)=\operatorname{Var}(\varrho)=\lim_{k}\beta_{k}''(1)/(k\ln2)$,
and $\overline{\mathsf{V}}(\nu)=0$, $\beta_{k}''(1)=o(k)$ and $h(\nu)=D_{0}(\nu)$
are equivalent. Without convergence the converse implications fail
(Remark~\ref{rem:generation-variance}): the four dimensions are
extremal values that do not see spread at intermediate scales, and
$\beta_{k}''(1)=o(k)$ is strictly weaker than $\overline{\mathsf{V}}(\nu)=0$;
the threshold, not the generation, is the consistent index for a variance
of the quantization defect.

After fixing the notation, we state and prove the mass-threshold formula
(Section~\ref{sec:stopping}), including the dual problem, the Tauberian
form and the entropy-per-level formulation. The general local-dimension
bounds and the spectral enclosure follow without local convergence
(Sections~\ref{sec:local-principle} and~\ref{sec:entropy}). Only
then do we impose convergence or regularity and derive the equality
and comparison results (Section~\ref{sec:convergence}), followed
by the Bernoulli examples (Section~\ref{sec:bernoulli-mixtures}).
Section~\ref{sec:literature} records the relation with earlier work,
the provenance of each statement, and the distinction between dimension
identities and quantization-coefficient estimates.

\subsection{Relation with earlier results}

\label{sec:literature}

Graf--Luschgy \cite[Example~4.1]{GL04} prove the harmonic rule for
two measures under logarithmic moment assumptions, Zhu \cite[Lemma~2.2]{Zhu12}
records lower and upper harmonic bounds, and \cite[Proposition~3.8]{Zhu15}
the finite-component rule; the mixture statements used here are attributed
reformulations. The mass-threshold antichain $\{\sigma:p_{\sigma^{-}}\geq t>p_{\sigma}\}$
occurs in \cite[Section~5, Lemma~5.7]{GL04}, and the ratio $s_{j}=\sum_{\sigma\in\Lambda_{j}}p_{\sigma}\log p_{\sigma}\big/\sum_{\sigma\in\Lambda_{j}}p_{\sigma}\log\|f_{\sigma}'\|$
of \cite[Section~4]{Zhu12}, as well as the antichain estimates of
\cite[Lemma~2.7]{Zhu15}, are direct predecessors of the ratio $\mathcal{H}_{\nu}(t)/\mathcal{L}_{\nu}(t)$.
The finite recurrent self-similar dimension formula of Roychowdhury--Snigireva
\cite{RS15} is an instance of Corollary~\ref{cor:KN-differentiable}:
in the strongly separated, irreducible setting with positive stationary
entropy the $L^{q}$-spectrum is differentiable at $1$ and $\dim_{\infty}>0$,
so \cite[Theorem~1.5(iii)]{KNneg} gives $D_{0}=h=-\beta'(1)$; the
same applies to the settings of \cite{RS21} whenever these two hypotheses
hold. Finally, a dimension identity only gives $\mathfrak{e}_{N,0}=N^{-1/D+o(1)}$;
the coefficient statements $0<\liminf_{N}N^{1/D}\mathfrak{e}_{N,0}\leq\limsup_{N}N^{1/D}\mathfrak{e}_{N,0}<\infty$
of \cite[Theorem~5.11]{GL04} and \cite[Theorem~4.1]{Zhu12}, and
the convergence-order results of \cite{Zhu15}, are strictly stronger
and are not subsumed by the dimension-level consequences of the mass-threshold
formula or of the negative-order theorems.

\subsection{Notation and standing hypothesis}

\label{sec:notation}

Let $\nu$ be a compactly supported Borel probability measure on $\R^{d}$.
Distances are taken in the maximum norm. Equivalent norms multiply
quantization errors by factors bounded independently of $N$, and
hence do not change their dimensions.

For $N\geq1$, the geometric-mean quantization error is 
\begin{equation}
\mathfrak{e}_{N,0}(\nu)\coloneqq\inf_{1\leq\#A\leq N}\exp\!\left(\int\ln d(x,A)\d\nu(x)\right).\label{eq:e0}
\end{equation}
The infimum is over non-empty finite sets $A\subset\R^{d}$, with
$d(x,A)\coloneqq\inf_{a\in A}\|x-a\|_{\infty}$. The natural logarithm
and exponential in \eqref{eq:e0} give the usual geometric mean \cite{GL04,KNneg}.
For comparison with neighbouring orders, set 
\[
\mathfrak{e}_{N,r}(\nu)\coloneqq\inf_{1\leq\#A\leq N}\left(\int d(x,A)^{r}\d\nu(x)\right)^{1/r},\qquad r\ne0.
\]
For any order $r$, including $r=0$, for which the errors are positive
and finite, define 
\begin{equation}
\underline{D}_{r}(\nu)\coloneqq\liminf_{N\to\infty}\frac{\log N}{-\log\mathfrak{e}_{N,r}(\nu)},\qquad\overline{D}_{r}(\nu)\coloneqq\limsup_{N\to\infty}\frac{\log N}{-\log\mathfrak{e}_{N,r}(\nu)}.\label{eq:quant-dims}
\end{equation}
Their common value, when it exists, is denoted by $D_{r}(\nu)$. The
ratio is independent of the logarithm base. For the order-zero proofs
we use the binary logarithmic error 
\begin{equation}
B_{\nu}(N)\coloneqq-\log\mathfrak{e}_{N,0}(\nu)=\sup_{1\leq\#A\leq N}\int-\log d(x,A)\d\nu(x),\qquad N\geq1.\label{eq:Bnu}
\end{equation}
Thus reciprocating the normalised growth rates $B_{\nu}(N)/\log N$
interchanges lower and upper limits in \eqref{eq:quant-dims}.

Let $\D_{k}$ denote the half-open dyadic cubes of side length $2^{-k}$.
Zero-mass cubes are omitted from mass sums and stopping partitions;
they are retained explicitly in full-tree cardinality identities.
Put 
\[
H_{k}(\nu)\coloneqq-\sum_{Q\in\D_{k}}\nu(Q)\log\nu(Q),\qquad h_{k}(\nu)\coloneqq\frac{H_{k}(\nu)}{k}.
\]
We write 
\[
\underline{h}(\nu)\coloneqq\liminf_{k\to\infty}h_{k}(\nu),\qquad\overline{h}(\nu)\coloneqq\limsup_{k\to\infty}h_{k}(\nu).
\]

When $\underline{h}(\nu)=\overline{h}(\nu)$, their common value is
denoted by $h(\nu)$. These fixed-scale entropy dimensions are Rényi
information dimensions \cite{Renyi59}.  Two entropy-like quantities
of Section~\ref{sec:adaptive-spectrum} are indexed by a mass threshold
$t$ instead of a generation $k$: $\mathcal{H}_{\nu}(t)$, the entropy
of the stopping partition, and the black-board $\mathds{h}_{\nu}(t)$,
the arithmetic mean of the local dimension estimate produced by the
adaptive algorithm, the threshold analogue of $h_{k}(\nu)$.

For $q\in\R$, define the finite-scale dyadic $L^{q}$-spectrum by
\begin{equation}
\beta_{k}(q)\coloneqq\frac{1}{k}\log\sum_{Q\in\D_{k},\,\nu(Q)>0}\nu(Q)^{q}.\label{eq:finite-beta}
\end{equation}
Thus $\beta_{k}(1)=0$. We use the spectrum, when the limit exists,
and its upper version:
\[
\beta(q)\coloneqq\lim_{k\to\infty}\beta_{k}(q),\qquad\overline{\beta}(q)\coloneqq\limsup_{k\to\infty}\beta_{k}(q).
\]

The ordinary-spectrum approach to positive-order quantization is developed
in \cite{KNZ}, and the negative-order partition-function framework
in \cite{KNneg}.

The upper spectrum is finite and convex on $(0,\infty)$ (see the
proof of Proposition~\ref{prop:entropy-only}), so its one-sided
derivatives at $1$ exist. Put
\begin{equation}
\sigma_{-}\coloneqq-\overline{\beta}'_{+}(1),\qquad\sigma_{+}\coloneqq-\overline{\beta}'_{-}(1),\qquad J_{\nu}\coloneqq-\partial\overline{\beta}(1)=[\sigma_{-},\sigma_{+}],\label{eq:reflected-subdifferential}
\end{equation}
the reflected subdifferential of $\overline{\beta}$ at $1$. It is
a single point exactly when $\overline{\beta}$ is differentiable
at $1$.

For later comparison, let 
\begin{equation}
\dim_{*}(\nu)\coloneqq\inf_{\nu(E)>0}\dim_{H}E,\qquad\dim^{*}(\nu)\coloneqq\inf_{\nu(E)=1}\dim_{H}E.\label{eq:measure-H-definitions}
\end{equation}
The infima run over Borel sets. Replacing Hausdorff dimension by packing
dimension defines $\mathrm{Dim}_{*}(\nu)$ and $\mathrm{Dim}^{*}(\nu)$;
see \cite[Definition~2.1]{Heu07}. These are dimensions of positive-
or full-mass sets, not dimensions of the topological support.

For the dyadic stopping and spectral results, normalise the measure
so that its support lies in the interior of the root cube $\mathcal{Q}=[0,1]^{d}$;
this keeps all mass away from the outer faces of the half-open cubes.
Use a nested, disjoint partition of $\mathcal{Q}$, assigning outer
boundary faces to the adjacent end cubes. Thus $\D_{0}=\{\mathcal{Q}\}$,
$\#\D_{k}=2^{dk}$, and no zero-boundary-mass assumption is needed.
Write $Q_{k}(x)$ for the unique cube containing $x$, and $\ell(Q)=2^{-k}=\operatorname{diam}Q$
for $Q\in\D_{k}$.

For $x\in\mathcal{Q}$ and $k\geq1$ put
\begin{equation}
I_{k}(x)\coloneqq-\log\nu(Q_{k}(x)),\qquad X_{k}(x)\coloneqq\frac{I_{k}(x)}{k},\label{eq:local-information}
\end{equation}
and define the lower and upper dyadic local dimensions
\begin{equation}
\underline{s}(x)\coloneqq\liminf_{k\to\infty}X_{k}(x),\qquad\overline{s}(x)\coloneqq\limsup_{k\to\infty}X_{k}(x).\label{eq:local-exponents}
\end{equation}
When they coincide we write $s(x)$ for the common value. These are
explicitly dyadic information exponents; their relation with the usual
ball local dimensions is recorded in Lemma~\ref{lem:balls-restrictions}.

Whenever a uniform mass hypothesis is required below, it is the following
condition: 
\begin{equation}
\nu(Q)\leq C2^{-ak}\quad(Q\in\D_{k},\ k\geq0),\qquad a>0,\quad C\geq1.\label{eq:uniform-dyadic-mass}
\end{equation}
Existence of such $a,C$ is equivalent to 
\begin{equation}
\dim_{\infty}(\nu)\coloneqq\liminf_{k\to\infty}\frac{-\log\max_{Q\in\D_{k}}\nu(Q)}{k}>0.\label{eq:positive-infinity-dimension}
\end{equation}
Indeed, every $0<a<\dim_{\infty}(\nu)$ satisfies \eqref{eq:uniform-dyadic-mass}
after enlarging $C$ to cover finitely many generations. Conversely,
\eqref{eq:uniform-dyadic-mass} implies $\dim_{\infty}(\nu)\geq a$.
The hypothesis implies that $\nu$ has no atoms. We always discard
the null set on which some $Q_{k}(x)$ has zero mass. The same exponent
ensures positive finite neighbouring-order errors for $-a<r<0$; order-zero
finiteness is proved in Lemma~\ref{lem:stopping-codebook}.

\section{Mass-adaptive stopping partitions}

\label{sec:stopping}

Assume \eqref{eq:uniform-dyadic-mass}. Fixed-generation entropy averages
information at a prescribed spatial scale. The order-zero analogue
fixes a mass threshold and averages the depth required to reach it.
Such antichains are classical in self-similar and self-conformal quantization
\cite[Section~5]{GL04}, \cite[Section~4]{Zhu12}. We prove a general
dyadic characterisation for both lower and upper dimensions, using
the estimate \eqref{eq:arbitrary-codebook-bound} for arbitrary codebooks
rather than a model-specific separation argument. The stopping partitions
are the output of the adaptive approximation algorithm of \cite{KN25};
the precise dictionary is given in Remark~\ref{rem:adaptive-algorithm}.

\subsection{Stopping quantities and the main theorem}

For $0<t<1$, let $Q^{-}$ denote the immediate dyadic parent of $Q$
and define 
\begin{equation}
P_{\nu}(t)\coloneqq\{Q\in\D\setminus\D_{0}:0<\nu(Q)<t\leq\nu(Q^{-})\}.\label{eq:mass-stopping-partition}
\end{equation}
Logarithmic thresholds are written out as $\log(1/t)$; where a symbol
is needed we abbreviate $\ell_{t}\coloneqq\log(1/t)$, so that $t=2^{-\ell_{t}}$
and $\ell_{t}$ increases from $0$ to $\infty$ as $t\downarrow0$.
Set 
\begin{equation}
M_{\nu}(t)\coloneqq\#P_{\nu}(t),\qquad\mathcal{L}_{\nu}(t)\coloneqq-\sum_{Q\in P_{\nu}(t)}\nu(Q)\log\ell(Q).\label{eq:stopping-cardinality-def}
\end{equation}
Recall the defintion in \prettyref{eq:stopping-time} for the stopping
depth $\tau_{t}\left(x\right)$. By Lemma~\ref{lem:stopping-size}
below, $P_{\nu}(t)$ is a finite partition modulo a $\nu$-null set.
Consequently, 
\begin{equation}
\mathcal{L}_{\nu}(t)=\int\tau_{t}(x)\d\nu(x).\label{eq:depth-as-integral}
\end{equation}

The pointwise counterpart of the ratio in Theorem~\ref{thm:adaptive-entropy}
below is the dimension estimate produced by the adaptive algorithm
at threshold $t$,
\begin{equation}
\varsigma_{t}(x)\coloneqq\frac{\log(1/t)}{\tau_{t}(x)},\label{eq:threshold-estimate}
\end{equation}
so that, by \eqref{eq:depth-as-integral}, $\log(1/t)/\mathcal{L}_{\nu}(t)=\bigl(\int\varsigma^{-1}_{t}\d\nu\bigr)^{-1}$
is the harmonic mean of $\varsigma_{t}$. In terms of the information
rates, $\varsigma_{t}$ is the rate stopped at $\tau_{t}$ and corrected
for the overshoot at the crossing: since $I_{\tau_{t}(x)-1}(x)\leq\log(1/t)<I_{\tau_{t}(x)}(x)$,
\begin{equation}
\varsigma_{t}(x)=X_{\tau_{t}(x)}(x)-\frac{I_{\tau_{t}(x)}(x)-\log(1/t)}{\tau_{t}(x)},\qquad0\leq X_{\tau_{t}(x)}(x)-\varsigma_{t}(x)<\frac{I_{\tau_{t}(x)}(x)-I_{\tau_{t}(x)-1}(x)}{\tau_{t}(x)}.\label{eq:stopped-rate}
\end{equation}
The correction is what makes $\varsigma_{t}$, rather than the stopped
rate itself, the right object: the stopping times omit the generations
at which $k\mapsto I_{k}(x)$ does not increase, so the lower limit
of $X_{\tau_{t}(x)}(x)$ as $t\downarrow0$ can exceed $\underline{s}(x)$,
whereas $\varsigma_{t}(x)=\log(1/t)/\tau_{t}(x)$ runs through all
intermediate values and satisfies \eqref{eq:first-crossing} below.
Since $k\mapsto I_{k}(x)$ is non-decreasing and $\tau_{t}(x)$ is
its generalised inverse, $\tau_{t}(x)=\#\{k\geq0:I_{k}(x)\leq\log(1/t)\}$,
and therefore
\begin{equation}
\liminf_{t\downarrow0}\varsigma_{t}(x)=\underline{s}(x),\qquad\limsup_{t\downarrow0}\varsigma_{t}(x)=\overline{s}(x).\label{eq:first-crossing}
\end{equation}
In particular $\varsigma_{t}(x)\to s(x)$ whenever the local dimension
$s(x)$ exists.

\begin{theorem}\label{thm:adaptive-entropy}Assume \eqref{eq:uniform-dyadic-mass}.
Then 
\begin{equation}
\underline{D}_{0}(\nu)=\liminf_{t\downarrow0}\frac{\log(1/t)}{\mathcal{L}_{\nu}(t)},\qquad\overline{D}_{0}(\nu)=\limsup_{t\downarrow0}\frac{\log(1/t)}{\mathcal{L}_{\nu}(t)}.\label{eq:threshold-lower}
\end{equation}
Both values lie in the reflected subdifferential $J_{\nu}=[\sigma_{-},\sigma_{+}]$
of \eqref{eq:reflected-subdifferential}. In both formulas one may
replace $t\downarrow0$ by $t=N^{-c}$ for any fixed $c>0$.

\end{theorem}

In particular, the choice $t=1/N$ and \eqref{eq:depth-as-integral}
give the stopping-time formulation
\begin{equation}
\underline{D}_{0}(\nu)=\liminf_{N\to\infty}\frac{\log N}{\int\tau_{1/N}\d\nu},\qquad\overline{D}_{0}(\nu)=\limsup_{N\to\infty}\frac{\log N}{\int\tau_{1/N}\d\nu}.\label{eq:threshold-time-lower}
\end{equation}
The numerator is a prescribed information threshold; the denominator
is the mean logarithmic spatial depth needed to cross it.

\begin{remark}[The adaptive approximation algorithm]\label{rem:adaptive-algorithm}In
the terminology of \cite{KN25}, $P_{\nu}(t)$ is the minimal $t$-good
partition $G_{1/t}$ for $\mathfrak{J}=\nu$ restricted to the dyadic
cubes of positive mass. The dual problem of \cite[Section~1.2]{KN25},
minimising the largest mass over partitions with at most $N$ cubes,
enters the lower bound \eqref{eq:stopping-centres} below through
the condition $M_{\nu}(t)\leq N$; its order-zero counterpart, maximising
the mean depth over partitions of bounded cardinality, is solved by
the threshold partitions as well (Proposition~\ref{prop:mean-depth-dual}).
The results of \cite{KN25} concern the growth exponent of the cardinality
$M_{\nu}(t)$, which equals one for $\mathfrak{J}=\nu$ by \cite[Corollary~1.6]{KN25}
and is made quantitative in Lemma~\ref{lem:stopping-size}; Theorem~\ref{thm:adaptive-entropy}
uses instead the $\nu$-mean depth $\mathcal{L}_{\nu}(t)=\int\tau_{t}\d\nu$
of the same partition, a first moment of the running time of the algorithm
rather than a count of its output. For $r\neq0$ the quantization
dimensions of order $r$ are obtained in \cite{KNZ,KNneg} from the
counting exponents for $\mathfrak{J}_{\nu,r}$; Lemma~\ref{lem:critical-zero}
and Theorem~\ref{thm:spectral-enclosure} connect those to the order-zero
quantities through the ordinary upper spectrum.\end{remark}

\begin{remark}[Atomlessness alone is insufficient]\label{rem:ahlfors-necessary}The
hypothesis \eqref{eq:uniform-dyadic-mass} cannot be weakened to the
absence of atoms. Let $d=1$ and let $\nu$ be the probability measure
on $[0,1)$ with $\nu([0,2^{-k}))=1/k$ for all $k\geq1$, spread
uniformly over each interval $R_{k}\coloneqq[2^{-k-1},2^{-k})$; the
map $x\mapsto(1+x)/2$ moves its support into the interior of the
unit interval and shifts every generation by one, which changes nothing
below. Then $\nu$ has no atoms, $\dim_{\infty}(\nu)=0$ by \eqref{eq:positive-infinity-dimension},
and
\[
\int-\log x\d\nu(x)\geq\sum_{k\geq1}k\,\nu(R_{k})=\sum_{k\geq1}\frac{1}{k+1}=\infty,
\]
so $B_{\nu}(N)=\infty$ and $\mathfrak{e}_{N,0}(\nu)=0$ for every
$N$; formally, \eqref{eq:quant-dims} gives $\underline{D}_{0}(\nu)=\overline{D}_{0}(\nu)=0$
with $1/\infty\coloneqq0$. On the other hand, for $x\in R_{k}$ we
have $\nu(Q_{j}(x))=1/j$ for $j\leq k$ and $\nu(Q_{j}(x))\leq2^{k+1-j}$
for $j>k$, hence $\tau_{t}(x)\leq2+\log(1/t)+\min\{k,1/t\}$. Since
$\nu(R_{k})=1/(k(k+1))$,
\[
\mathcal{L}_{\nu}(t)\leq2+\log(1/t)+\sum_{k\leq1/t}\frac{1}{k+1}+\frac{1}{t}\,\nu\bigl([0,2^{-\lfloor1/t\rfloor-1})\bigr)\leq(1+\ln2)\log(1/t)+4,
\]
so that $\liminf_{t\downarrow0}\log(1/t)/\mathcal{L}_{\nu}(t)\geq1/(1+\ln2)>0$
(in fact the ratio converges to $1/(1+\ln2)$), and \eqref{eq:threshold-lower}
fails. What breaks down is the upper bound \eqref{eq:arbitrary-codebook-bound},
whose proof uses \eqref{eq:uniform-dyadic-mass} in the tail estimate
\eqref{eq:ball-uniform-mass}.\end{remark}

\subsection{Cardinality and depth estimates}

\begin{lemma}[Size and depth of the stopping partition]\label{lem:stopping-size}Assume
\eqref{eq:uniform-dyadic-mass} and put
\begin{equation}
K(t)\coloneqq1+\frac{\log(1/t)+\log C}{a}\qquad(0<t<1).\label{eq:max-stopping-depth}
\end{equation}
\begin{enumerate}\item For every $0<t<1$, every cube of $P_{\nu}(t)$
has generation at most $K(t)$, and $\tau_{t}\leq K(t)$ on $\mathcal{Q}$.
In particular, $P_{\nu}(t)$ is a finite partition modulo a $\nu$-null
set. Moreover, with constants $C_{1},C_{2}$ depending only on $a,C,d$,
\begin{equation}
t^{-1}<M_{\nu}(t)\leq C_{1}t^{-1}\bigl(1+\log(1/t)\bigr),\qquad\frac{\log(1/t)}{d}\leq\mathcal{L}_{\nu}(t)\leq\frac{\log(1/t)}{a}+C_{2}.\label{eq:stopping-cardinality}
\end{equation}
\item With $\sigma_{\pm}$ as in \eqref{eq:reflected-subdifferential},
one has $a\leq\sigma_{-}$ and
\begin{equation}
\frac{1}{\sigma_{+}}\leq\liminf_{t\downarrow0}\frac{\mathcal{L}_{\nu}(t)}{\log(1/t)}\leq\limsup_{t\downarrow0}\frac{\mathcal{L}_{\nu}(t)}{\log(1/t)}\leq\frac{1}{\sigma_{-}}.\label{eq:depth-sigma-bounds}
\end{equation}
\end{enumerate}

\end{lemma}
\begin{proof}
(1) If $Q\in P_{\nu}(t)$ has generation $k$, then $t\leq\nu(Q^{-})\leq C2^{-a(k-1)}$,
so $k\leq K(t)$. Likewise $\nu(Q_{k}(x))\leq C2^{-ak}<t$ for every
$x$ and every $k>K(t)-1$, so $\tau_{t}\leq K(t)$. Thus every positive-mass
dyadic path crosses the threshold by generation $K(t)$; the first-crossing
cubes are disjoint and cover $\nu$-almost every point, so $P_{\nu}(t)$
is a finite partition modulo a null set.

Since each stopping cube has mass less than $t$, their total mass
one implies $M_{\nu}(t)>t^{-1}$. At each generation there are at
most $t^{-1}$ cubes of mass at least $t$. Each has at most $2^{d}$
children, and every stopping cube is one of these children. Summing
over the at most $K(t)$ possible parent generations gives $M_{\nu}(t)\leq2^{d}t^{-1}K(t)$,
which is the upper bound in \eqref{eq:stopping-cardinality} with
$C_{1}\coloneqq2^{d}\max\{1+(\log C)/a,\,1/a\}$. Integrating $\tau_{t}\leq K(t)$
gives the upper depth bound $\mathcal{L}_{\nu}(t)\leq K(t)=\log(1/t)/a+C_{2}$
with $C_{2}\coloneqq1+(\log C)/a$. For the lower depth bound, the
stopping masses sum to one, while disjointness gives $\sum_{Q\in P_{\nu}(t)}\ell(Q)^{d}\leq1$.
Jensen's inequality for the exponential function therefore yields
\[
2^{-d\mathcal{L}_{\nu}(t)}=2^{\sum_{Q\in P_{\nu}(t)}\nu(Q)\log\ell(Q)^{d}}\leq\sum_{Q\in P_{\nu}(t)}\nu(Q)\ell(Q)^{d}\leq t\sum_{Q\in P_{\nu}(t)}\ell(Q)^{d}\leq t.
\]
Taking logarithms gives $\mathcal{L}_{\nu}(t)\geq\log(1/t)/d$.

(2) We first show that, for $0<q<1<q'$, $\overline{\beta}(q')\leq-a(q'-1)<0<\overline{\beta}(q)$
and
\begin{equation}
\frac{1-q}{\overline{\beta}(q)}\leq\liminf_{t\downarrow0}\frac{\mathcal{L}_{\nu}(t)}{\log(1/t)}\leq\limsup_{t\downarrow0}\frac{\mathcal{L}_{\nu}(t)}{\log(1/t)}\leq\frac{1-q'}{\overline{\beta}(q')}.\label{eq:depth-spectral-bounds}
\end{equation}
Since $\tau_{t}$ takes values in $\{1,2,\dots\}$, \eqref{eq:depth-as-integral}
gives $\mathcal{L}_{\nu}(t)=\sum_{k\geq0}\nu\{\tau_{t}>k\}$, and
since $k\mapsto\nu(Q_{k}(x))$ is non-increasing,
\begin{equation}
\{\tau_{t}>k\}=\{x:\nu(Q_{k}(x))\geq t\}=\{x:X_{k}(x)\leq\ell_{t}/k\}\qquad(k\geq1).\label{eq:stopping-as-level-set}
\end{equation}
For $0<\theta<1$ and $c\in\R$, Markov's inequality applied to $\nu(Q_{k}(\cdot))^{\pm\theta}$
together with \eqref{eq:finite-beta} gives
\begin{equation}
\nu\{X_{k}\leq c\}\leq2^{k(\beta_{k}(1+\theta)+\theta c)},\qquad\nu\{X_{k}\geq c\}\leq2^{k(\beta_{k}(1-\theta)-\theta c)},\label{eq:markov-window}
\end{equation}
so that, taking $c=\ell_{t}/k$ in \eqref{eq:stopping-as-level-set},
\[
\nu\{\tau_{t}>k\}\leq2^{k\beta_{k}(1+\theta)+\theta\ell_{t}},\qquad\nu\{\tau_{t}\leq k\}\leq2^{k\beta_{k}(1-\theta)-\theta\ell_{t}}.
\]
The bound $\sum_{Q\in\D_{k}}\nu(Q)^{1+\theta}\leq(\max_{Q\in\D_{k}}\nu(Q))^{\theta}\leq(C2^{-ak})^{\theta}$
gives $\overline{\beta}(1+\theta)\leq-a\theta<0$ and hence $a\leq\sigma_{-}$,
while convexity of $\overline{\beta}$ with $\overline{\beta}(1)=0$
gives $\overline{\beta}(1-\theta)\geq\theta\sigma_{-}>0$; both quotients
in \eqref{eq:depth-spectral-bounds} are therefore finite and positive.

\emph{Upper bound.} Put $q'=1+\theta$ and let $0<\eta<-\overline{\beta}(q')$.
There is $k_{\eta}$ with $\beta_{k}(q')\leq\overline{\beta}(q')+\eta$
for $k\geq k_{\eta}$; put $\Lambda\coloneqq-\overline{\beta}(q')-\eta>0$.
For such $k$ we get $\nu\{\tau_{t}>k\}\leq2^{-k\Lambda+\theta\ell_{t}}$,
which for $k=\lceil\theta\ell_{t}/\Lambda\rceil+j$, $j\geq0$, is
at most $2^{-\Lambda j}$. Bounding the remaining terms by $1$,
\[
\mathcal{L}_{\nu}(t)\leq k_{\eta}+\frac{\theta\ell_{t}}{\Lambda}+1+\sum_{j\geq0}2^{-\Lambda j}=\frac{\theta}{-\overline{\beta}(q')-\eta}\,\log(1/t)+O_{\theta,\eta}(1).
\]
Dividing by $\ell_{t}$, letting $t\downarrow0$ and then $\eta\downarrow0$
gives the upper bound in \eqref{eq:depth-spectral-bounds}, since
$\theta/(-\overline{\beta}(q'))=(1-q')/\overline{\beta}(q')$.

\emph{Lower bound.} Put $q=1-\theta$ and let $\eta>0$. There is
$k_{\eta}$ with $\beta_{k}(q)\leq\overline{\beta}(q)+\eta$ for $k\geq k_{\eta}$,
so that $\nu\{\tau_{t}\leq k\}\leq2^{k(\overline{\beta}(q)+\eta)-\theta\ell_{t}}$.
For $k_{\eta}\leq k\leq k_{\ast}\coloneqq\lfloor\theta\ell_{t}/(\overline{\beta}(q)+2\eta)\rfloor$
the exponent is at most $-\eta k$, whence
\[
\mathcal{L}_{\nu}(t)\geq\sum^{k_{\ast}}_{k=k_{\eta}}\nu\{\tau_{t}>k\}\geq k_{\ast}-k_{\eta}-\sum_{j\geq0}2^{-\eta j}=\frac{\theta}{\overline{\beta}(q)+2\eta}\,\log(1/t)-O_{\theta,\eta}(1).
\]
Dividing by $\ell_{t}$ and letting $t\downarrow0$, then $\eta\downarrow0$,
gives the lower bound.

Finally, by convexity of $\overline{\beta}$ and $\overline{\beta}(1)=0$
the quotient $(1-q')/\overline{\beta}(q')$ decreases to $1/\sigma_{-}$
as $q'\downarrow1$, and $(1-q)/\overline{\beta}(q)$ increases to
$1/\sigma_{+}$ as $q\uparrow1$; this is \eqref{eq:depth-sigma-bounds}.
\end{proof}

\subsection{Codebook estimates in both directions}

\begin{lemma}[Comparison with arbitrary codebooks]\label{lem:stopping-codebook}
Assume \eqref{eq:uniform-dyadic-mass} and write $B(N)=B_{\nu}(N)$.
For every $0<t<1$ and $N\geq1$, 
\begin{align}
B(M_{\nu}(t)) & \geq\mathcal{L}_{\nu}(t),\label{eq:stopping-centres}\\
B(N) & \leq\mathcal{L}_{\nu}(t)+C_{3}Nt\bigl(1+\log(1/t)\bigr),\label{eq:arbitrary-codebook-bound}
\end{align}
where $C_{3}$ depends only on $a,C,d$. In particular $B(N)<\infty$
and $\mathfrak{e}_{N,0}(\nu)>0$ for every finite $N$.

\end{lemma}
\begin{proof}
For \eqref{eq:stopping-centres}, let $A_{t}$ consist of the centres
of the cubes in $P_{\nu}(t)$, so that $\#A_{t}=M_{\nu}(t)$. For
$x\in Q\in P_{\nu}(t)$ the maximum-norm distance from $x$ to the
centre of $Q$ is at most $\ell(Q)/2$, hence $d(x,A_{t})\leq\ell(Q)$
and $-\log d(x,A_{t})\geq-\log\ell(Q)=k(Q)$. Integrating over $Q$
and summing over $P_{\nu}(t)$,
\[
B(M_{\nu}(t))\geq\int-\log d(x,A_{t})\d\nu(x)\geq\sum_{Q\in P_{\nu}(t)}\nu(Q)\,k(Q)=\mathcal{L}_{\nu}(t).
\]
The remaining inequality \eqref{eq:arbitrary-codebook-bound} is proved
in four steps.

\emph{Step 1 (normalisation).} Fix a non-empty codebook $A$ with
$\#A\leq N$ and put $f_{A}(x)\coloneqq-\log d(x,A)$. Coordinatewise
projection onto $\mathcal{Q}$ does not increase the maximum-norm
distance from any point of $\mathcal{Q}$, and does not increase the
cardinality, so we may assume $A\subset\mathcal{Q}$. Since $\mathcal{Q}$
has maximum-norm diameter $1$, this gives $d(x,A)\leq1$ and hence
\begin{equation}
f_{A}\geq0\qquad\text{on }\mathcal{Q}.\label{eq:fA-nonnegative}
\end{equation}
This is the only point at which the normalisation is used, and it
is what makes the tail sum below an upper bound.

\emph{Step 2 (from the integral to the sets $U_{k}(A)$).} For $k\geq0$
put
\[
U_{k}(A)\coloneqq\{x\in\mathcal{Q}:d(x,A)<2^{-k}\}=\{x\in\mathcal{Q}:f_{A}(x)>k\},
\]
so that $U_{0}(A)\supseteq U_{1}(A)\supseteq\cdots$. A point $x$
belongs to $U_{k}(A)$ for exactly those $k\geq0$ with $k<f_{A}(x)$,
that is, for exactly $\lceil f_{A}(x)\rceil$ values of $k$. By \eqref{eq:fA-nonnegative}
we therefore have the pointwise identity $\sum_{k\geq0}\mathbf{1}_{U_{k}(A)}=\lceil f_{A}\rceil\geq f_{A}$,
and monotone convergence gives
\begin{equation}
\int-\log d(x,A)\d\nu(x)\leq\sum^{\infty}_{k=0}\nu(U_{k}(A)).\label{eq:tail-sum-bound}
\end{equation}

\emph{Step 3 (two bounds for $\nu(U_{k}(A))$).} A maximum-norm ball
of radius $2^{-k}$ is a cube of side $2^{1-k}$ and therefore meets
at most $c_{d}\coloneqq3^{d}$ cubes of $\D_{k}$. Since $U_{k}(A)$
is covered by the $\#A\leq N$ balls of radius $2^{-k}$ centred at
the points of $A$, it meets at most $c_{d}N$ cubes of $\D_{k}$;
call this family $\mathcal{F}_{k}$.

If a cube $Q\in\D_{k}$ contains a point $x$ with $\tau_{t}(x)\leq k$,
then $Q=Q_{k}(x)\subseteq Q_{\tau_{t}(x)}(x)$ and the latter has
mass less than $t$ by \eqref{eq:stopping-time}; hence $\nu(Q)<t$.
Splitting $U_{k}(A)$ according to whether stopping has occurred by
generation $k$ and bounding the cubes of $\mathcal{F}_{k}$ that
meet $\{\tau_{t}\leq k\}$ one by one,
\begin{equation}
\nu(U_{k}(A))\leq\nu\{\tau_{t}>k\}+c_{d}Nt.\label{eq:ball-stopping-split}
\end{equation}
Independently, \eqref{eq:uniform-dyadic-mass} bounds the mass of
each cube of $\mathcal{F}_{k}$ by $C2^{-ak}$, so
\begin{equation}
\nu(U_{k}(A))\leq c_{d}CN2^{-ak}.\label{eq:ball-uniform-mass}
\end{equation}
These combine to
\begin{equation}
\nu(U_{k}(A))\leq\nu\{\tau_{t}>k\}+c_{d}N\min\{t,C2^{-ak}\}:\label{eq:ball-combined}
\end{equation}
if $t\leq C2^{-ak}$ the minimum is $t$ and \eqref{eq:ball-stopping-split}
is the assertion; if $C2^{-ak}<t$ the minimum is $C2^{-ak}$ and
\eqref{eq:ball-uniform-mass} already gives the assertion without
the first term.

\emph{Step 4 (summation).} Inserting \eqref{eq:ball-combined} into
\eqref{eq:tail-sum-bound},
\[
\int-\log d(x,A)\d\nu(x)\leq\sum^{\infty}_{k=0}\nu\{\tau_{t}>k\}+c_{d}N\sum^{\infty}_{k=0}\min\{t,C2^{-ak}\}.
\]
Since $\tau_{t}$ takes values in $\{1,2,\dots\}$, the first sum
is $\int\tau_{t}\d\nu=\mathcal{L}_{\nu}(t)$ by \eqref{eq:depth-as-integral}.
For the second, let $k_{0}$ be the least $k\geq0$ with $C2^{-ak}\leq t$,
so that $k_{0}\leq1+\log(C/t)/a$ and $C2^{-ak_{0}}\leq t$. Splitting
the sum at $k_{0}$,
\[
\sum^{\infty}_{k=0}\min\{t,C2^{-ak}\}\leq k_{0}t+\sum_{k\geq k_{0}}C2^{-ak}\leq k_{0}t+\frac{t}{1-2^{-a}}\leq C_{4}\,t\bigl(1+\log(1/t)\bigr),
\]
with $C_{4}$ depending only on $a$ and $C$. Taking the supremum
over all non-empty $A$ with $\#A\leq N$ proves \eqref{eq:arbitrary-codebook-bound}
with $C_{3}=c_{d}C_{4}$. Finally $\mathcal{L}_{\nu}(t)<\infty$ by
\eqref{eq:stopping-cardinality}, so $B(N)<\infty$ and $\mathfrak{e}_{N,0}(\nu)>0$
for every finite $N$.
\end{proof}

\subsection{Proof of the mass-threshold formula}

The following lemma is standard; it allows us to pass between the
continuous parameter $t$ and the discrete parameter $N$ without
any continuity assumption. Call a sequence $(a_{n})_{n\geq1}$ of
positive numbers \emph{admissible} if it is non-decreasing, $a_{n}\to\infty$,
and $\sup_{n}a_{n+1}/a_{n}<\infty$. Applied to $g=\log f$ for a
non-decreasing $f\geq1$, the lemma below is the usual statement that
growth exponents may be computed along admissible sequences; its proof
is straightforward.

\begin{lemma}[Admissible sequences]\label{lem:admissible-sequence}Let
$(a_{n})$ be admissible and let $g:[a_{1},\infty)\to[0,\infty)$
be non-decreasing. Then
\begin{equation}
\liminf_{n\to\infty}\frac{g(a_{n})}{\log a_{n}}=\liminf_{x\to\infty}\frac{g(x)}{\log x},\qquad\limsup_{n\to\infty}\frac{g(a_{n})}{\log a_{n}}=\limsup_{x\to\infty}\frac{g(x)}{\log x}.\label{eq:admissible-sampling}
\end{equation}
\end{lemma}

With this lemma at hand we can prove the theorem.
\begin{proof}[Proof of Theorem~\ref{thm:adaptive-entropy}]
Fix $\varepsilon>0$. By \eqref{eq:stopping-cardinality}, for all
sufficiently large $N$,
\[
M_{\nu}\bigl(N^{-1/(1+\varepsilon)}\bigr)\leq N.
\]
Monotonicity of $B$ therefore gives
\begin{equation}
B(N)\geq\mathcal{L}_{\nu}\bigl(N^{-1/(1+\varepsilon)}\bigr).\label{eq:stopping-B-lower}
\end{equation}
In \eqref{eq:arbitrary-codebook-bound}, instead take $t=N^{-(1+\varepsilon)}$.
Since $Nt(1+\log(1/t))=O(N^{-\varepsilon}\log N)=o(1)$,
\begin{equation}
B(N)\leq\mathcal{L}_{\nu}\bigl(N^{-(1+\varepsilon)}\bigr)+o(1).\label{eq:stopping-B-upper}
\end{equation}
The function $x\mapsto\mathcal{L}_{\nu}(1/x)$ is non-decreasing,
because lowering the threshold cannot shorten a stopping path, and
for every $\delta>0$ the sequence $a_{N}=N^{\delta}$ is admissible.
Lemma~\ref{lem:admissible-sequence} therefore gives
\begin{equation}
\delta\,\liminf_{t\downarrow0}\frac{\mathcal{L}_{\nu}(t)}{\log(1/t)}=\liminf_{N\to\infty}\frac{\mathcal{L}_{\nu}(N^{-\delta})}{\log N},\qquad\delta\,\limsup_{t\downarrow0}\frac{\mathcal{L}_{\nu}(t)}{\log(1/t)}=\limsup_{N\to\infty}\frac{\mathcal{L}_{\nu}(N^{-\delta})}{\log N}.\label{eq:sampling-delta}
\end{equation}
Applying \eqref{eq:sampling-delta} with $\delta=1/(1+\varepsilon)$
to \eqref{eq:stopping-B-lower} and with $\delta=1+\varepsilon$ to
\eqref{eq:stopping-B-upper}, and letting $\varepsilon\downarrow0$,
we obtain
\begin{equation}
\liminf_{t\downarrow0}\frac{\mathcal{L}_{\nu}(t)}{\log(1/t)}=\liminf_{N\to\infty}\frac{B(N)}{\log N},\qquad\limsup_{t\downarrow0}\frac{\mathcal{L}_{\nu}(t)}{\log(1/t)}=\limsup_{N\to\infty}\frac{B(N)}{\log N}.\label{eq:depth-B-limits}
\end{equation}
Moreover, by \eqref{eq:depth-sigma-bounds},
\[
0<\frac{1}{\sigma_{+}}\leq\liminf_{t\downarrow0}\frac{\mathcal{L}_{\nu}(t)}{\log(1/t)}\leq\limsup_{t\downarrow0}\frac{\mathcal{L}_{\nu}(t)}{\log(1/t)}\leq\frac{1}{\sigma_{-}}<\infty.
\]
All quantities are positive and finite. Reciprocation therefore interchanges
lower and upper limits in \eqref{eq:depth-B-limits} and, by \eqref{eq:quant-dims},
yields
\[
\underline{D}_{0}(\nu)=\liminf_{t\downarrow0}\frac{\log(1/t)}{\mathcal{L}_{\nu}(t)},\qquad\overline{D}_{0}(\nu)=\limsup_{t\downarrow0}\frac{\log(1/t)}{\mathcal{L}_{\nu}(t)},
\]
which is \eqref{eq:threshold-lower}, together with $\sigma_{-}\leq\underline{D}_{0}(\nu)\leq\overline{D}_{0}(\nu)\leq\sigma_{+}$.
The sampling statement is \eqref{eq:sampling-delta} with $\delta=c$,
and $c=1$ together with \eqref{eq:depth-as-integral} gives \eqref{eq:threshold-time-lower}.
\end{proof}

\subsection{The dual problem at order zero and a Tauberian form}

We now transfer the counting framework of \cite{KN25} to the mean
depth; Remark~\ref{rem:adaptive-algorithm} records the dictionary.
For $\mathfrak{J}=\nu$ and $x=1/t$, the cubes with $\nu(Q)\geq t$
are the $x$-bad cubes $R_{x}$ of \cite[proof of Proposition~3.1]{KN25},
and $P_{\nu}(t)$ is the minimal $x$-good partition $G_{x}$ of \cite[(1.1)]{KN25}
up to null cubes. Let $\Pi$ be the set of finite partitions of the
root cube into cubes from $\D$; for $P\in\Pi$ let $S(P)\subset\D$
be the set of cubes strictly containing a member of $P$, the cubes
subdivided in the construction of $P$, and with $k(Q)$ the generation
of $Q$ put
\[
\mathcal{L}(P)\coloneqq\sum_{Q\in P}\nu(Q)\,k(Q),\qquad\Gamma_{\nu}(N)\coloneqq\sup\{\mathcal{L}(P):P\in\Pi,\ \#P\leq N\},
\]
the $\nu$-mean depth of $P$ and the largest mean depth attainable
with at most $N$ cubes, and
\[
S_{t}\coloneqq\{Q\in\D:\nu(Q)\geq t\},\qquad N_{t}\coloneqq1+(2^{d}-1)\,\#S_{t}\qquad(0<t<1),
\]
so that $S_{t}=R_{1/t}$.

\begin{proposition}[Weighted count and the dual problem at order zero]\label{prop:mean-depth-dual}Assume
\eqref{eq:uniform-dyadic-mass} and let $0<t<1$.\begin{enumerate}\item
The mean depth is the $\nu$-weighted number of bad cubes,
\begin{equation}
\mathcal{L}_{\nu}(t)=\sum_{Q\in\D,\ \nu(Q)\geq t}\nu(Q)=\sum_{Q\in S_{t}}\nu(Q),\label{eq:weighted-count}
\end{equation}
and $\#S_{t}\leq t^{-1}K(t)$ with $K(t)$ from \eqref{eq:max-stopping-depth}.\item
For every $P\in\Pi$,
\begin{equation}
\#P=1+(2^{d}-1)\,\#S(P),\qquad\mathcal{L}(P)=\sum_{Q\in S(P)}\nu(Q).\label{eq:tree-identities}
\end{equation}
\item The family $\widetilde{P}_{\nu}(t)\coloneqq\{Q\in\D\setminus\D_{0}:\nu(Q)<t\leq\nu(Q^{-})\}$,
the partition $G_{1/t}$ of \cite[(1.1)]{KN25} with null cubes retained,
consists of $P_{\nu}(t)$ and the null children of cubes in $S_{t}$;
it belongs to $\Pi$ and satisfies $S(\widetilde{P}_{\nu}(t))=S_{t}$,
$\#\widetilde{P}_{\nu}(t)=N_{t}\geq M_{\nu}(t)$ and $\mathcal{L}(\widetilde{P}_{\nu}(t))=\mathcal{L}_{\nu}(t)$.
Every $P\in\Pi$ with $\#P\leq N_{t}$ satisfies $\mathcal{L}(P)\leq\mathcal{L}_{\nu}(t)$.
Consequently $\Gamma_{\nu}(N_{t})=\mathcal{L}_{\nu}(t)$: the threshold
partitions solve the mean-depth dual problem.\item For all $N\geq1$
and $0<t<1$,
\begin{equation}
\Gamma_{\nu}(N)\leq B_{\nu}(N)\leq\Gamma_{\nu}(N_{t})+C_{3}Nt\bigl(1+\log(1/t)\bigr),\label{eq:dual-sandwich}
\end{equation}
and
\begin{equation}
\underline{D}_{0}(\nu)=\liminf_{N\to\infty}\frac{\log N}{\Gamma_{\nu}(N)},\qquad\overline{D}_{0}(\nu)=\limsup_{N\to\infty}\frac{\log N}{\Gamma_{\nu}(N)}.\label{eq:dual-dimension}
\end{equation}
\end{enumerate}\end{proposition}

Item (3) is the order-zero analogue of the dual problem of \cite[Lemma~1.3]{KN25},
and \eqref{eq:dual-dimension} is the order-zero counterpart of \cite[Theorem~1.8]{KN25}.
\begin{proof}
For (1), $k\mapsto\nu(Q_{k}(x))$ is non-increasing, so $\tau_{t}(x)=\#\{k\geq0:\nu(Q_{k}(x))\geq t\}$
by \eqref{eq:stopping-time}; integrating with \eqref{eq:depth-as-integral}
gives \eqref{eq:weighted-count}, and as in the proof of Lemma~\ref{lem:stopping-size}
only generations $k<K(t)$ contribute, with at most $t^{-1}$ cubes
each. For (2), $P\in\Pi$ is the leaf set of a finite tree of dyadic
cubes with internal nodes $S(P)$; replacing a cube $Q$ of generation
$k$ by its $2^{d}$ children adds one node to $S(P)$, $2^{d}-1$
leaves to $P$ and $\sum_{Q'}\nu(Q')(k+1)-\nu(Q)k=\nu(Q)$ to $\mathcal{L}(P)$,
so induction on $\#S(P)$ proves \eqref{eq:tree-identities}; for
the threshold partition the first identity is the exact form of $\#(G_{x}\cap\D_{k})\leq2^{d}\,\#(R_{x}\cap\D_{k-1})$
in the proof of \cite[Proposition~3.1]{KN25}.

For (3), $S_{t}$ is finite by (1) and closed under passing to the
parent, so it is the set of internal nodes of a tree as in (2) whose
leaves are the cubes with mass below $t$ and parent of mass at least
$t$, that is, $\widetilde{P}_{\nu}(t)$; this is the description
of $G_{1/t}$ in \cite[Lemma~1.2]{KN25}. Hence $\widetilde{P}_{\nu}(t)\in\Pi$,
$S(\widetilde{P}_{\nu}(t))=S_{t}$, $\#\widetilde{P}_{\nu}(t)=N_{t}$
and $\mathcal{L}(\widetilde{P}_{\nu}(t))=\mathcal{L}_{\nu}(t)$; the
null members contribute nothing and the others form $P_{\nu}(t)$,
so $M_{\nu}(t)\leq N_{t}$. If $P\in\Pi$ and $\#P\leq N_{t}$, then
$\#S(P)\leq\#S_{t}$; every cube in $S(P)\setminus S_{t}$ has mass
below $t$, every cube in $S_{t}\setminus S(P)$ has mass at least
$t$, and the second set is at least as large as the first, so $\mathcal{L}(P)=\sum_{S(P)}\nu\leq\sum_{S_{t}}\nu=\mathcal{L}_{\nu}(t)$.
This exchange of cubes replaces the merging step in the proof of \cite[Lemma~1.2]{KN25}.

For (4), let $P\in\Pi$ with $\#P\leq N$ and let $A$ be the set
of centres of the cubes of positive mass in $P$. Then $\#A\leq N$
and $d(x,A)\leq\ell(Q)$ for $x\in Q\in P$, so $B_{\nu}(N)\geq\mathcal{L}(P)$
by \eqref{eq:Bnu}, as in \cite{KNZ,KNneg}, where the centres of
an optimal partition serve as a codebook; the supremum over $P$ gives
the lower bound in \eqref{eq:dual-sandwich}, and the upper bound
is \eqref{eq:arbitrary-codebook-bound} combined with (3). Since $B_{\nu}(N)\geq\Gamma_{\nu}(N)$,
the lower and upper limits of $\log N/\Gamma_{\nu}(N)$ are at least
$\underline{D}_{0}(\nu)$ and $\overline{D}_{0}(\nu)$. For the reverse
inequalities put $a_{N}\coloneqq N/\log^{2}N$ and $t_{N}\coloneqq1/a_{N}$
for $N\geq8$, an admissible sequence in the sense of Lemma~\ref{lem:admissible-sequence}.
By (1) and \eqref{eq:max-stopping-depth}, $N_{t_{N}}\leq2^{d}a_{N}K(t_{N})+1\leq N$
for large $N$, so (3) gives $\Gamma_{\nu}(N)\geq\mathcal{L}_{\nu}(t_{N})$
and
\[
\frac{\log N}{\Gamma_{\nu}(N)}\leq\frac{\log N}{\log a_{N}}\cdot\frac{\log a_{N}}{\mathcal{L}_{\nu}(t_{N})}.
\]
Here $\log N/\log a_{N}\to1$, and Lemma~\ref{lem:admissible-sequence},
applied to $x\mapsto\mathcal{L}_{\nu}(1/x)$, shows that the lower
and upper limits of $\log a_{N}/\mathcal{L}_{\nu}(t_{N})$ are $\underline{D}_{0}(\nu)$
and $\overline{D}_{0}(\nu)$ by Theorem~\ref{thm:adaptive-entropy}.
This proves \eqref{eq:dual-dimension}.
\end{proof}

The series $\sum_{Q\in\D}\nu(Q)^{q}$ has abscissa of convergence
$\kappa_{\nu}=1$, the critical exponent of \cite[Section~1.4 and Lemma~2.4]{KN25}.
The mean depth is its Laplace--Stieltjes counting function, and a
Tauberian theorem extracts the dimensions from the behaviour of the
series at $\kappa_{\nu}$.

\begin{corollary}[Tauberian form of the mass-threshold formula]\label{cor:tauberian}Assume
\eqref{eq:uniform-dyadic-mass} and put
\[
Z_{\nu}(q)\coloneqq\sum_{Q\in\D}\nu(Q)^{q}=\sum_{k\geq0}2^{k\beta_{k}(q)},\qquad q>1,
\]
which is finite. With $F_{\nu}(u)\coloneqq\sum_{Q\in\mathcal{D},-\log\nu(Q)\le u}\nu(Q)$,
for $u\in\mathbb{R}$, we have $F_{\nu}(u)=\mathcal{L}_{\nu}(2^{-u})$
for $u>0$ and $F_{\nu}(u)=0$ for $u<0$. In this way we define a
non-decreasing right-continuous function and hence a positive Lebesgue--Stieltjes
measure, including its atom at zero. Then
\begin{equation}
Z_{\nu}(1+\lambda)=\int_{[0,\infty)}2^{-\lambda u}\d F_{\nu}(u)\qquad(\lambda>0),\label{eq:laplace-identity}
\end{equation}
and
\begin{equation}
\frac{1}{\overline{D}_{0}(\nu)}\leq\liminf_{q\downarrow1}(q-1)\ln2\,Z_{\nu}(q)\leq\limsup_{q\downarrow1}(q-1)\ln2\,Z_{\nu}(q)\leq\frac{1}{\underline{D}_{0}(\nu)}.\label{eq:tauberian-bounds}
\end{equation}
The dimension $D_{0}(\nu)$ exists if and only if $(q-1)Z_{\nu}(q)$
converges as $q\downarrow1$, and then
\begin{equation}
D_{0}(\nu)=\Bigl(\lim_{q\downarrow1}(q-1)\ln2\sum_{Q\in\D}\nu(Q)^{q}\Bigr)^{-1}.\label{eq:tauberian-formula}
\end{equation}
\end{corollary}
\begin{proof}
Finiteness follows from $\sum_{Q\in\D_{k}}\nu(Q)^{1+\lambda}\leq(C2^{-ak})^{\lambda}$.
Write $F_{k}(u)\coloneqq\nu\{I_{k}\leq u\}$. Then $F_{\nu}=\sum_{k\geq0}F_{k}$
by \eqref{eq:weighted-count} ; on every bounded interval this sum
is finite by the mass bound. Each $F_{k}$ is a right-continuous distribution
function, and $\int_{[0,\infty)}2^{-\lambda u}\d F_{k}(u)=\int2^{-\lambda I_{k}}\d\nu=\sum_{Q\in\D_{k}}\nu(Q)^{1+\lambda}$.
Summing proves \eqref{eq:laplace-identity}.

Put $\mu\coloneqq\lambda\ln2$ and $\widehat{\mathcal{L}}(\mu)\coloneqq\int_{[0,\infty)}e^{-\mu u}\d F_{\nu}(u)=Z_{\nu}(1+\lambda)$.
Integration by parts gives
\[
\mu\widehat{\mathcal{L}}(\mu)=\mu^{2}\int^{\infty}_{0}e^{-\mu u}F_{\nu}(u)\d u=\int^{\infty}_{0}\frac{F_{\nu}(u)}{u}\,\mu^{2}ue^{-\mu u}\d u,
\]
an average of $\mathcal{L}_{\nu}(t)/\ell_{t}$ with respect to a probability
density whose mass leaves every bounded interval as $\mu\downarrow0$.
Since $\mathcal{L}_{\nu}(t)/\ell_{t}$ is bounded for large threshold
depth by \eqref{eq:stopping-cardinality}, $\liminf_{\mu\downarrow0}\mu\widehat{\mathcal{L}}(\mu)\geq g_{-}$
and $\limsup_{\mu\downarrow0}\mu\widehat{\mathcal{L}}(\mu)\leq g_{+}$,
where $g_{-}\coloneqq\liminf_{t\downarrow0}\mathcal{L}_{\nu}(t)/\log(1/t)=1/\overline{D}_{0}(\nu)$
and $g_{+}\coloneqq\limsup_{t\downarrow0}\mathcal{L}_{\nu}(t)/\log(1/t)=1/\underline{D}_{0}(\nu)$
by Theorem~\ref{thm:adaptive-entropy}. The contribution from any
fixed bounded interval tends to zero because the counting function
is locally bounded. This is \eqref{eq:tauberian-bounds}, and it shows
that the limit in \eqref{eq:tauberian-formula} exists whenever $D_{0}(\nu)$
does. Conversely, if $\mu\widehat{\mathcal{L}}(\mu)\to c$ as $\mu\downarrow0$,
Karamata's Tauberian theorem \cite[Theorem~1.7.1]{BGT87}, applied
to the non-decreasing right-continuous function $F_{\nu}$, gives
$\mathcal{L}_{\nu}(t)\sim c\log(1/t)$ as $t\downarrow0$; thus $D_{0}(\nu)$
exists and equals $1/c$, where $c\geq1/\sigma_{+}>0$ by \eqref{eq:tauberian-bounds}.
\end{proof}

\begin{remark}[Pointwise form and residue]\label{rem:pointwise-tauberian}\label{rem:residue}For
fixed $x$ put $T_{x}(u)\coloneqq\#\{k\geq0:I_{k}(x)\leq u\}$ for
$u\geq0$ and $T_{x}(u)\coloneqq0$ for $u<0$, so that $\tau_{t}(x)=T_{x}(\ell_{t})$;
$T_{x}$ is non-decreasing and right-continuous. Then
\[
\int_{[0,\infty)}2^{-\lambda u}\d T_{x}(u)=\sum_{k\geq0}\nu(Q_{k}(x))^{\lambda}\qquad(\lambda>0),
\]
and \eqref{eq:laplace-identity} is the $\nu$-integral of this identity.
Since the lower and upper limits of $\varsigma_{t}(x)$ are $\underline{s}(x)$
and $\overline{s}(x)$ by \eqref{eq:first-crossing}, and $\tau_{t}(x)\leq K(t)$,
the proof of Corollary~\ref{cor:tauberian} applies pointwise and
gives
\[
\frac{1}{\overline{s}(x)}\leq\liminf_{\lambda\downarrow0}\lambda\ln2\sum_{k\geq0}\nu(Q_{k}(x))^{\lambda}\leq\limsup_{\lambda\downarrow0}\lambda\ln2\sum_{k\geq0}\nu(Q_{k}(x))^{\lambda}\leq\frac{1}{\underline{s}(x)}:
\]
the dyadic local dimension $s(x)$ exists and is finite if and only
if $\lambda\sum_{k}\nu(Q_{k}(x))^{\lambda}$ converges to a positive
limit as $\lambda\downarrow0$, and then $s(x)$ is the reciprocal
of the product of $\ln2$ and this limit. Corollary~\ref{cor:tauberian}
is the $\nu$-integrated form of this criterion, as the harmonic formula
\eqref{eq:local-arithmetic-formula} is the $\nu$-integrated form
of \eqref{eq:inverse-local-growth}; the integrated criterion does
not require the local dimension to exist almost everywhere. The transform
averages $\mathcal{L}_{\nu}(2^{-u})/u$ against the Gamma density
$\mu^{2}ue^{-\mu u}$, where $\mu=(q-1)\ln2$ . The Tauberian theorem
transfers existence of the limit; the bounds alone do not identify
separate lower and upper limits. The abscissa $\kappa_{\nu}=1$ fixes
the location of the singularity of $Z_{\nu}$, and the dimension at
order zero is carried by the rate of blow-up there, in the manner
of a residue; since $Z_{\nu}(q)=\sum_{k}2^{k\beta_{k}(q)}$, this
rate depends on the finite-generation spectra near $q=1$ and not
only on $\overline{\beta}$, from which Theorem~\ref{thm:spectral-enclosure}
extracts only the enclosure $[\sigma_{-},\sigma_{+}]$ (Proposition~\ref{prop:spectral-window}).
\end{remark}

\subsection{Entropy per level and the finite-threshold harmonic mean}

\label{sec:adaptive-spectrum}

Theorem~\ref{thm:adaptive-entropy} uses only the mean depth. The
entropy of the stopping partition and the entropy per unit depth,
\begin{equation}
\mathcal{H}_{\nu}(t)\coloneqq-\sum_{Q\in P_{\nu}(t)}\nu(Q)\log\nu(Q),\qquad R_{\nu}(t)\coloneqq\frac{\mathcal{H}_{\nu}(t)}{\mathcal{L}_{\nu}(t)},\label{eq:stopping-entropy-def}
\end{equation}
give an equivalent formulation of the theorem and an exact finite-scale
identity that anticipates the harmonic formula of Remark~\ref{cor:local-harmonic}:
the quantization dimension of order zero is the entropy per level
of refinement of the adaptive tree, whereas the entropy dimension
is the entropy per level of the uniform tree.

\begin{lemma}[Entropy per level]\label{prop:harmonic-stopping}Assume
\eqref{eq:uniform-dyadic-mass}, let $0<t<1$, and let $C_{1}\geq1$
be as in \eqref{eq:stopping-cardinality}. Then
\begin{equation}
\log(1/t)\leq\mathcal{H}_{\nu}(t)\leq\log M_{\nu}(t)\leq\log(1/t)+\log\!\bigl(C_{1}(1+\log(1/t))\bigr),\label{eq:stopping-entropy-asymptotic}
\end{equation}
consequently
\begin{equation}
0\leq R_{\nu}(t)-\frac{\log(1/t)}{\mathcal{L}_{\nu}(t)}\leq d\,\frac{\log\!\bigl(C_{1}(1+\log(1/t))\bigr)}{\log(1/t)}\longrightarrow0,\label{eq:threshold-replacement-error}
\end{equation}
and Theorem~\ref{thm:adaptive-entropy} takes the equivalent form
\begin{equation}
\underline{D}_{0}(\nu)=\liminf_{t\downarrow0}R_{\nu}(t),\qquad\overline{D}_{0}(\nu)=\limsup_{t\downarrow0}R_{\nu}(t),\label{eq:adaptive-lower}
\end{equation}
where $t\downarrow0$ may be replaced by $t=N^{-c}$ for any fixed
$c>0$.\end{lemma}
\begin{proof}
Every stopping mass is less than $t$ and these masses sum to one,
hence $\mathcal{H}_{\nu}(t)\geq\log(1/t)$; by Jensen's inequality
the entropy is at most the logarithm of the number of atoms, and \eqref{eq:stopping-cardinality}
gives the upper bounds in \eqref{eq:stopping-entropy-asymptotic}.
Dividing $0\leq\mathcal{H}_{\nu}(t)-\log(1/t)\leq\log(C_{1}(1+\log(1/t)))$
by $\mathcal{L}_{\nu}(t)\geq\log(1/t)/d$ proves \eqref{eq:threshold-replacement-error},
and \eqref{eq:adaptive-lower} with the sampling statement follows
from Theorem~\ref{thm:adaptive-entropy}. No lower bound on the ratios
$\nu(Q)/t$ is assumed.
\end{proof}

Moreover, $R_{\nu}(t)$ is an exact weighted harmonic mean of the
cube-wise dimensions: with $s_{Q}\coloneqq\log\nu(Q)/\log\ell(Q)>0$
and $w_{Q}(t)\coloneqq\nu(Q)(-\log\nu(Q))/\mathcal{H}_{\nu}(t)$ for
$Q\in P_{\nu}(t)$, so that $\sum_{Q}w_{Q}(t)=1$, the identity $-\log\ell(Q)=(-\log\nu(Q))/s_{Q}$
gives $\mathcal{L}_{\nu}(t)/\mathcal{H}_{\nu}(t)=\sum_{Q}w_{Q}(t)/s_{Q}$,
that is,
\[
R_{\nu}(t)=\Bigl(\sum_{Q\in P_{\nu}(t)}\frac{w_{Q}(t)}{s_{Q}}\Bigr)^{-1},\qquad\min_{Q}s_{Q}\leq R_{\nu}(t)\leq\max_{Q}s_{Q}.
\]

The dimension estimate $\varsigma_{t}$ of \eqref{eq:threshold-estimate}
is the cube-wise dimension $s_{Q}=X_{\tau_{t}(x)}(x)$ of the stopping
cube $Q\ni x$ up to the overshoot, see \eqref{eq:stopped-rate}.
For $0<t<1$ put
\begin{equation}
\mathds{h}_{\nu}(t)\coloneqq\int\varsigma_{t}\d\nu,\qquad\mathsf{D}_{\nu}(t)\coloneqq\frac{\log(1/t)}{\mathcal{L}_{\nu}(t)}=\Bigl(\int\varsigma^{-1}_{t}\d\nu\Bigr)^{-1},\qquad\mathsf{V}_{\nu}(t)\coloneqq\operatorname{Var}_{\nu}(\varsigma_{t}),\label{eq:threshold-means}
\end{equation}
its arithmetic mean, its harmonic mean, which is the quantity of Theorem~\ref{thm:adaptive-entropy},
and its variance. The generation-indexed counterparts are $X_{k}$,
the entropy per level $h_{k}(\nu)=\int X_{k}\d\nu$ with lower and
upper limits $\underline{h}(\nu),\overline{h}(\nu)$, and $\operatorname{Var}_{\nu}(X_{k})=\beta_{k}''(1)/(k\ln2)$.
Note that $\mathds{h}_{\nu}(t)$ is a mean of dimension estimates,
whereas $\mathcal{H}_{\nu}(t)$ in \eqref{eq:stopping-entropy-def}
is an entropy of order $\log(1/t)$; the relation between $\mathds{h}_{\nu}(t)$
and $\underline{h}(\nu),\overline{h}(\nu)$ is given in (3) and (4)
below and in \eqref{eq:threshold-arithmetic-bounds}. We call
\[
\overline{\mathsf{V}}(\nu)\coloneqq\limsup_{t\downarrow0}\mathsf{V}_{\nu}(t)
\]
the upper threshold variance of the local dimension. It is finite
by (1) below, requires no convergence of any kind, and is the variance
of the limit law under the hypothesis of Proposition~\ref{prop:spectral-window}.

\begin{proposition}[Threshold means and the upper threshold variance]\label{prop:threshold-variance}Assume
\eqref{eq:uniform-dyadic-mass}, let $0<t<1$, and put $a_{t}\coloneqq\log(1/t)/K(t)$,
so that $a_{t}\to a$ as $t\downarrow0$.\begin{enumerate}\item
$a_{t}\leq\varsigma_{t}\leq\log(1/t)$, and $\nu\{\varsigma_{t}\geq v\}\leq t^{1-d/v}$
for $v>d$. Consequently $\int((\varsigma_{t}-d)^{+})^{n}\d\nu\to0$
as $t\downarrow0$ for every $n\geq1$, all moments of $\varsigma_{t}$
are bounded for $t\leq1/2$, and
\begin{equation}
\mathsf{D}_{\nu}(t)=\Bigl(\int\varsigma^{-1}_{t}\d\nu\Bigr)^{-1}\leq\mathds{h}_{\nu}(t),\qquad a_{t}\leq\mathsf{D}_{\nu}(t),\qquad\mathds{h}_{\nu}(t)\leq d+o(1).\label{eq:threshold-harmonic}
\end{equation}
\item The defect has the exact representations
\begin{equation}
\mathds{h}_{\nu}(t)-\mathsf{D}_{\nu}(t)=\int\frac{(\varsigma_{t}-\mathsf{D}_{\nu}(t))^{2}}{\varsigma_{t}}\d\nu=\frac{\mathsf{D}_{\nu}(t)}{\mathds{h}_{\nu}(t)}\int\frac{(\varsigma_{t}-\mathds{h}_{\nu}(t))^{2}}{\varsigma_{t}}\d\nu,\label{eq:threshold-defect}
\end{equation}
and
\begin{equation}
\frac{\mathsf{D}_{\nu}(t)}{\mathds{h}_{\nu}(t)}\Bigl(\frac{\mathsf{V}_{\nu}(t)}{d}-\varepsilon_{t}\Bigr)\leq\mathds{h}_{\nu}(t)-\mathsf{D}_{\nu}(t)\leq\frac{\mathsf{D}_{\nu}(t)}{\mathds{h}_{\nu}(t)}\cdot\frac{\mathsf{V}_{\nu}(t)}{a_{t}},\label{eq:threshold-variance-bounds}
\end{equation}
where $\varepsilon_{t}\coloneqq d^{-2}\int(\varsigma_{t}-\mathds{h}_{\nu}(t))^{2}(\varsigma_{t}-d)^{+}\d\nu\to0$
as $t\downarrow0$. In particular $\overline{\mathsf{V}}(\nu)=0$
if and only if $\mathds{h}_{\nu}(t)-\mathsf{D}_{\nu}(t)\to0$ as $t\downarrow0$.\item
If $\overline{\mathsf{V}}(\nu)=0$, then $\beta_{k}''(1)=o(k)$ and
\[
\underline{h}(\nu)=\liminf_{t\downarrow0}\mathds{h}_{\nu}(t)=\underline{D}_{0}(\nu),\qquad\overline{h}(\nu)=\limsup_{t\downarrow0}\mathds{h}_{\nu}(t)=\overline{D}_{0}(\nu).
\]
\item If the laws of $X_{k}$ under $\nu$ converge weakly to $\varrho$,
then the laws of $\varsigma_{t}$ converge weakly to $\varrho$ as
$t\downarrow0$, and
\begin{equation}
\begin{aligned}\mathds{h}_{\nu}(t) & \to\int y\d\varrho(y)=h(\nu),\qquad\mathsf{D}_{\nu}(t)\to\Bigl(\int\frac{1}{y}\d\varrho(y)\Bigr)^{-1}=D_{0}(\nu),\\
\mathsf{V}_{\nu}(t) & \to\operatorname{Var}(\varrho)=\overline{\mathsf{V}}(\nu)=\lim_{k\to\infty}\frac{\beta_{k}''(1)}{k\ln2},
\end{aligned}
\label{eq:threshold-consistency}
\end{equation}
where the identifications on the right are those of Proposition~\ref{prop:spectral-window}.
In this case $\overline{\mathsf{V}}(\nu)=0$, $\beta_{k}''(1)=o(k)$
and $h(\nu)=D_{0}(\nu)$ are equivalent, and \eqref{eq:threshold-defect}
becomes \eqref{eq:exact-gap-identity} in the limit.\end{enumerate}\end{proposition}
\begin{proof}
For (1), $\tau_{t}\geq1$ gives $\varsigma_{t}\leq\log(1/t)$, and
$\tau_{t}\leq K(t)$ by \eqref{eq:max-stopping-depth} gives $\varsigma_{t}\geq a_{t}$.
For $v>d$ let $k\coloneqq\lfloor\log(1/t)/v\rfloor$. Then $\varsigma_{t}(x)\geq v$
means $\tau_{t}(x)\leq k$, hence $\nu(Q_{k}(x))<t$, and the total
mass of the generation-$k$ cubes of mass below $t$ is at most $2^{dk}t\leq t^{1-d/v}$.
Writing $\int((\varsigma_{t}-d)^{+})^{n}\d\nu=\int^{\infty}_{0}nv^{n-1}\nu\{\varsigma_{t}>d+v\}\d v$,
the integrand vanishes for $v>\log(1/t)$, is at most $nv^{n-1}t^{v/(2d)}$
for $0<v\leq d$ because $v/(d+v)\geq v/(2d)$, and is at most $nv^{n-1}t^{1/2}$
for $v>d$; hence the integral is at most $n!\,(2d/\ln(1/t))^{n}+\log(1/t)^{n}t^{1/2}\to0$.
Together with $\varsigma_{t}\leq d+(\varsigma_{t}-d)^{+}$ this bounds
all moments. The identity in \eqref{eq:threshold-harmonic} is the
definition $\mathcal{L}_{\nu}(t)=\int\tau_{t}\d\nu$, the inequality
is Jensen's, $\mathsf{D}_{\nu}(t)\geq a_{t}$ follows from $\mathcal{L}_{\nu}(t)\leq K(t)$,
and $\mathds{h}_{\nu}(t)\leq d+\int(\varsigma_{t}-d)^{+}\d\nu$.

For (2), all moments of $\varsigma_{t}$ and of $\varsigma^{-1}_{t}$
are finite, so expanding the squares as in the proof of Theorem~\ref{thm:quantitative-gap},
with the law of $\varsigma_{t}$ in place of $\varrho$, gives $\int(\varsigma_{t}-\mathsf{D})^{2}\varsigma^{-1}_{t}\d\nu=\mathsf{h}-2\mathsf{D}+\mathsf{D}^{2}\int\varsigma^{-1}_{t}\d\nu=\mathsf{h}-\mathsf{D}$
and $\int(\varsigma_{t}-\mathsf{h})^{2}\varsigma^{-1}_{t}\d\nu=\mathsf{h}-2\mathsf{h}+\mathsf{h}^{2}/\mathsf{D}=(\mathsf{h}/\mathsf{D})(\mathsf{h}-\mathsf{D})$,
where $\mathsf{h}=\mathds{h}_{\nu}(t)$ and $\mathsf{D}=\mathsf{D}_{\nu}(t)$.
The upper bound in \eqref{eq:threshold-variance-bounds} uses $\varsigma^{-1}_{t}\leq a^{-1}_{t}$
in the second representation. For the lower bound, $1/y\geq1/d-(y-d)^{+}/d^{2}$
for all $y>0$, since this reduces to $(y-d)^{2}\geq0$ when $y>d$;
inserting this in the second representation gives the bound with $\varepsilon_{t}$,
and $\varepsilon_{t}\to0$ by the Cauchy--Schwarz inequality, the
bounded fourth moments and $\int((\varsigma_{t}-d)^{+})^{2}\d\nu\to0$
from (1). Since $a_{t}/(d+o(1))\leq\mathsf{D}_{\nu}(t)/\mathds{h}_{\nu}(t)\leq1$
by \eqref{eq:threshold-harmonic}, the bounds \eqref{eq:threshold-variance-bounds}
show that $\mathds{h}_{\nu}(t)-\mathsf{D}_{\nu}(t)\to0$ if and only
if $\mathsf{V}_{\nu}(t)\to0$.

For (3), put $\mathcal{L}\coloneqq\mathcal{L}_{\nu}(t)$. Since $\varsigma_{t}\geq a_{t}\to a>0$,
vanishing threshold variance implies $\varsigma^{-1}_{t}-\int\varsigma^{-1}_{t}\d\nu\to0$
in $L^{2}(\nu)$. As $\int\varsigma^{-1}_{t}\d\nu=\mathcal{L}/\ell_{t}$
is bounded above and away from zero, $\tau_{t}/\mathcal{L}\to1$ in
$L^{2}(\nu)$. First we prove concentration at every generation. Otherwise,
tightness and uniform integrability of $X^{2}_{k}$ provide a subsequence,
constants $\varepsilon,\delta>0$, and thresholds $u_{k}\in[ak/2,(d+1)k]$
such that both $\nu\{I_{k}\leq u_{k}-\varepsilon k\}\geq\delta$ and
$\nu\{I_{k}\geq u_{k}+\varepsilon k\}\geq\delta$. Choose $0<\eta<\min\{1/2,\delta\varepsilon/(4d)\}$
and set $k_{\pm}=k\pm\lfloor\eta k\rfloor$. Conditional entropy gives
$\int(I_{k}-I_{k_{-}})\d\nu\leq d\eta k$ and $\int(I_{k_{+}}-I_{k})\d\nu\leq d\eta k$.
Markov's inequality therefore gives, for large $k$, $\nu\{\tau_{2^{-u_{k}}}\leq k_{-}\}\geq\delta/2$
and $\nu\{\tau_{2^{-u_{k}}}>k_{+}\}\geq\delta/2$. The corresponding
values of $\varsigma_{2^{-u_{k}}}$ lie on opposite sides of $u_{k}/k_{-}$
and $u_{k}/k_{+}$, whose difference is bounded below by a positive
constant. This contradicts vanishing threshold variance. Hence $\operatorname{Var}_{\nu}(X_{k})\to0$
and $X_{k}-h_{k}\to0$ in probability along all generations. Next
set $k=\lfloor\mathcal{L}\rfloor$ and $k_{\pm}=\lfloor(1\pm\eta)\mathcal{L}\rfloor$.
Outside a set of probability tending to zero, $k_{-}<\tau_{t}\leq k_{+}$,
so $|I_{k}-\ell_{t}|\leq I_{k_{+}}-I_{k_{-}}$. The expectation of
this increment is at most $d(k_{+}-k_{-})$; Markov's inequality,
followed by $\eta\downarrow0$, yields $X_{k}-\ell_{t}/k\to0$ in
probability. Uniform integrability gives $h_{k}-\mathsf{D}_{\nu}(t)\to0$.
Thus $\underline{h}\leq\underline{D}_{0}\leq\overline{D}_{0}\leq\overline{h}$.
For the reverse outer inequalities fix $0<\varepsilon<a/2$ and put
$u^{\pm}_{k}=H_{k}(\nu)\pm\varepsilon k$. The already proved generation
concentration implies $\nu\{\tau_{2^{-u^{-}_{k}}}>k\}\to0$ and $\nu\{\tau_{2^{-u^{+}_{k}}}>k\}\to1$
. The uniform stopping-depth bound then gives $\mathcal{L}_{\nu}(2^{-u^{-}_{k}})\leq k+o(k)$
and $\mathcal{L}_{\nu}(2^{-u^{+}_{k}})\geq k(1-o(1))$. Consequently
$\mathsf{D}_{\nu}(2^{-u^{-}_{k}})\geq h_{k}-\varepsilon-o(1)$ and
$\mathsf{D}_{\nu}(2^{-u^{+}_{k}})\leq h_{k}+\varepsilon+o(1)$. Taking
the appropriate subsequences and then $\varepsilon\downarrow0$ proves
$\underline{D}_{0}\leq\underline{h}$ and $\overline{D}_{0}\geq\overline{h}$.
Part (2) identifies the limits of $\mathds{h}_{\nu}$, and $\beta_{k}''(1)/(k\ln2)=\operatorname{Var}_{\nu}(X_{k})\to0$
completes (3).

For (4), let $y>0$ be a continuity point of $W(c)\coloneqq\varrho((-\infty,c])$
and put $k\coloneqq\lceil\log(1/t)/y\rceil-1$, so that $k\to\infty$
and $\log(1/t)/k\to y$ as $t\downarrow0$. Since $\varsigma_{t}\leq y$
means $\tau_{t}\geq k+1$, that is, $I_{k}\leq\log(1/t)$ by \eqref{eq:stopping-time},
we have $\nu\{\varsigma_{t}\leq y\}=\nu\{X_{k}\leq\log(1/t)/k\}$,
and the squeeze between continuity points used in the proof of Proposition~\ref{prop:spectral-window}(3)
shows that this converges to $W(y)$. Hence the laws of $\varsigma_{t}$
converge weakly to $\varrho$, and by the bounded moments in (1) and
the bound $\varsigma^{-1}_{t}\leq a^{-1}_{t}$, the means of $\varsigma_{t}$,
$\varsigma^{2}_{t}$ and $\varsigma^{-1}_{t}$ converge to those of
$\varrho$; this proves \eqref{eq:threshold-consistency}, the identifications
of the limits being Proposition~\ref{prop:spectral-window}(3), and
the three equivalent conditions are Theorem~\ref{thm:quantitative-gap}.
The convergence $\mathsf{D}_{\nu}(t)\to(\int y^{-1}\d\varrho)^{-1}$,
combined with Theorem~\ref{thm:adaptive-entropy}, is a second proof
of the harmonic formula in that proposition.
\end{proof}

\begin{remark}[Generation-indexed variance and the limits of (3)]\label{rem:generation-variance}The
spectral counterpart of $\overline{\mathsf{V}}(\nu)$ is $\overline{V}(\nu)\coloneqq\limsup_{k}\beta_{k}''(1)/(k\ln2)=\limsup_{k}\operatorname{Var}_{\nu}(X_{k})$,
and $\overline{V}(\nu)=0$ if and only if $X_{k}-h_{k}(\nu)\to0$
in probability, which is weaker than almost-everywhere constancy of
the local dimension (Example~\ref{ex:law-not-ae}). Proposition~\ref{prop:threshold-variance}
shows that $\overline{\mathsf{V}}(\nu)=0$ implies $\overline{V}(\nu)=0$,
with equivalence under convergence in law; in general the threshold
is the consistent index for a variance of the quantization defect,
because the mean depth averages the laws of $X_{k}$ over a multiplicative
window of generations (proof of Proposition~\ref{prop:spectral-window}(3)).
One explicit construction shows both failures of converse. Fix $0<p<1/2$,
write $s\coloneqq h_{2}(p)$ and $b=-\log(1-p)$, and let $k_{i}=2^{2^{i}}$.
Set the binary digits in each run $(k_{i},2k_{i}]$ deterministically
to zero; all other digits are independent Bernoulli$(p)$. Let $A(k)$
count the non-deterministic digits up to $k$. Then $A(k)\geq k/3$
for all sufficiently large $k$, so the standing mass bound holds
with exponent $b/3$ after adjusting its constant. Also $I_{k}$ is
a sum of $A(k)$ independent copies of a bounded non-constant information
variable of mean $s$. Thus $\operatorname{Var}(X_{k})=O(1/k)$, while
almost surely $X_{k}-sA(k)/k\to0$ and the lower and upper local dimensions
are $s/2$ and $s$. At $u_{i}=sA(k_{i})$, the central limit theorem
splits the mass asymptotically in half according to whether $I_{k_{i}}>u_{i}$.
Bounded positive information increments at non-deterministic digits
and the flat deterministic run give $\tau_{2^{-u_{i}}}/k_{i}\Rightarrow\tfrac{1}{2}\delta_{1}+\tfrac{1}{2}\delta_{2}$.
Since $u_{i}/k_{i}\to s$, uniform integrability from Proposition
\ref{prop:threshold-variance}(1) yields $\mathsf{V}_{\nu}(2^{-u_{i}})\to s^{2}/16>0$.
The estimate is asymptotically two-valued, not exactly two-valued
at finite thresholds. At the shifted thresholds $u^{\pm}_{i}=u_{i}\pm k^{3/4}_{i}$,
concentration of the same sums gives $\tau_{2^{-u^{-}_{i}}}/k_{i}\to1$
and $\tau_{2^{-u^{+}_{i}}}/k_{i}\to2$ in probability and in mean,
using the uniform depth bound. Hence $\mathsf{D}_{\nu}(2^{-u^{-}_{i}})\to s$
and $\mathsf{D}_{\nu}(2^{-u^{+}_{i}})\to s/2$. The local harmonic
bounds give the reverse bounds on the extrema, so $\underline{h}=\underline{D}_{0}=s/2$
and $\overline{h}=\overline{D}_{0}=s$, although $\overline{\mathsf{V}}>0$.
Thus neither generation-variance decay nor equality of the lower and
upper dimension pairs implies vanishing threshold variance.\end{remark}
\begin{example}[A line--square mixture and its two spectra]
\label{ex:line-square}\label{ex:two-spectra}Let $E_{1}=[0,1/4]\times\{0\}$,
$E_{2}=[1/2,1]^{2}$ and $\nu=(\lambda_{E_{1}}+\lambda_{E_{2}})/2$
with normalised length and area, a variant of the interval--Cantor
mixture of \cite[Example~4.1]{GL04}. The local dimension is $1$
on $E_{1}$ and $2$ on $E_{2}$, so Remark~\ref{cor:local-harmonic}
gives $h(\nu)=3/2$ and $D_{0}(\nu)=4/3$. For $k\geq2$ the segment
contributes $2^{k-2}$ cubes of mass $2^{1-k}$ and the square $2^{2k-2}$
cubes of mass $2^{1-2k}$, so $H_{k}(\nu)=\tfrac{3}{2}k-1$, $\beta(q)=2(1-q)$
for $0<q\leq1$ and $\beta(q)=1-q$ for $q>1$: the ordinary spectrum
exists at every $q>0$ but has a corner at $1$, which is what leaves
room for $h(\nu)\neq D_{0}(\nu)$. By contrast, since $t^{q-1}\leq\sum_{Q\in P_{\nu}(t)}\nu(Q)^{q}\leq M_{\nu}(t)^{1-q}$
for $0<q<1$ and conversely for $q>1$, \eqref{eq:stopping-cardinality}
gives $\log\sum_{Q\in P_{\nu}(t)}\nu(Q)^{q}=(1-q)\log(1/t)+O(\log\log(1/t))$,
and with $\mathcal{L}_{\nu}(t)=\tfrac{3}{4}\log(1/t)+O(1)$ the $L^{q}$-sums
of the adaptive partition, normalised by $\mathcal{L}_{\nu}(t)$,
tend to the linear function $\tfrac{4}{3}(1-q)$, which agrees with
$\beta$ only at $q=1$.
\end{example}

\section{General local and classical dimension bounds}

\label{sec:local-principle}

We continue to assume \eqref{eq:uniform-dyadic-mass}. The local quantities
\eqref{eq:local-information}--\eqref{eq:local-exponents} are dyadic
information exponents; their use requires no assumption about the
mass of dyadic boundary hyperplanes.

\begin{lemma}[Uniform integrability at fixed generations]\label{lem:information-UI}
For every $\theta\in\R$ and every $n\geq1$ the families $(e^{-\theta X_{k}})_{k}$
and $(X^{n}_{k})_{k}$ are uniformly integrable, and the laws of the
$X_{k}$ are tight. Moreover, 
\[
a\leq\liminf_{k\to\infty}X_{k}(x)\leq\limsup_{k\to\infty}X_{k}(x)\leq d\quad\text{for \ensuremath{\nu}-almost every }x.
\]
The upper bound and uniform integrability hold for every probability
measure on the dyadic root, without \eqref{eq:uniform-dyadic-mass}.

\end{lemma}
\begin{proof}
For $v>d$, every generation-$k$ cube contributing to $\{X_{k}>v\}$
has mass less than $2^{-kv}$. There are at most $2^{dk}$ cubes,
so 
\begin{equation}
\nu\{X_{k}>v\}\leq2^{-k(v-d)}.\label{eq:information-tail}
\end{equation}
Since $X_{k}\geq0$, the laws of the $X_{k}$ are tight, and for fixed
$\theta\in\R$, $R>2d$ and $k\ln2\geq2|\theta|$, 
\[
\int e^{|\theta|X_{k}}\mathbf{1}_{\{X_{k}>R\}}\d\nu\leq e^{|\theta|R}\,2^{-k(R-d)}+|\theta|\int^{\infty}_{R}e^{|\theta|v}\,2^{-k(v-d)}\d v\leq2\cdot2^{-k(R/2-d)},
\]
which tends to zero as $R\to\infty$ uniformly in such $k$; the finitely
many remaining $k$ are harmless. The same estimate with $e^{|\theta|v}$
replaced by $v^{n}$ gives the uniform integrability of $(X^{n}_{k})_{k}$.
For every fixed $v>d$, \eqref{eq:information-tail} is summable in
$k$; Borel--Cantelli and a countable sequence $v\downarrow d$ give
the almost-sure upper bound. Finally, \eqref{eq:uniform-dyadic-mass}
implies $X_{k}\geq a-(\log C)/k$, proving the lower bound.
\end{proof}

\begin{lemma}[Balls, dyadic cubes, and normalised restrictions]\label{lem:balls-restrictions}
Using closed maximum-norm balls, for a probability measure $\mu$
on the dyadic root, the lower and upper dyadic local exponents equal
the corresponding ball local dimensions almost everywhere. If $E$
is Borel with $\mu(E)>0$, these exponents are unchanged almost everywhere
on $E$ when $\mu$ is replaced by $\mu_{E}=\mu|_{E}/\mu(E)$. Moreover,
a uniform dyadic bound with exponent $a>0$ is equivalent, up to constants,
to 
\begin{equation}
\mu(B(x,r))\leq C'r^{a}\qquad(x\in\R^{d},\ 0<r\leq1),\label{eq:uniform-ball-mass}
\end{equation}
and is inherited by $\mu_{E}$, with a constant allowed to depend
on $\mu(E)$.

\end{lemma}
\begin{proof}
A ball of radius comparable to $2^{-k}$ meets only a bounded number
of generation-$k$ cubes. Conversely, a cube lies in a ball of comparable
radius. This proves equivalence of the uniform bounds. For the almost-everywhere
comparison, enlarge each generation-$k$ cube $Q$ to the union $Q^{+}$
of the bounded family of its neighbouring cubes needed to contain
$B(x,2^{-k})$ for $x\in Q$. These enlargements have uniformly bounded
overlap. For $\varepsilon>0$, the total $\mu$-mass of cubes $Q\in\D_{k}$
satisfying $\mu(Q^{+})>2^{\varepsilon k}\mu(Q)$ is at most 
\[
2^{-\varepsilon k}\sum_{Q\in\D_{k}}\mu(Q^{+})\leq C_{d}2^{-\varepsilon k}.
\]
Borel--Cantelli gives almost everywhere $\limsup\left(\log\mu\left(Q^{+}\right)-\log\mu\left(Q\right)\right)/k\leq\varepsilon$.
Applied for rational $\varepsilon>0$, and the inclusions $Q_{k}(x)\subset B(x,2^{-k})\subset Q_{k}(x)^{+}$
show that the logarithmic masses differ by $o(k)$ almost everywhere.
Monotonicity in the radius extends the conclusion from dyadic radii
to all radii. Finally, dyadic conditional expectations give 
\[
\frac{\mu(E\cap Q_{k}(x))}{\mu(Q_{k}(x))}\longrightarrow1\quad\text{for almost every }x\in E.
\]
Consequently $\mu_{E}(Q_{k}(x))/\mu(Q_{k}(x))\to1/\mu(E)$ there,
which leaves both normalised logarithmic limits unchanged. The inherited
mass bound follows from $\mu_{E}(B)\leq\mu(B)/\mu(E)$.
\end{proof}

\begin{proposition}
[Integral bounds without local convergence]\label{prop:local-liminf-bounds}
Under \eqref{eq:uniform-dyadic-mass}, the dyadic local dimensions
$\underline{s},\overline{s}$ of \eqref{eq:local-exponents} satisfy
\begin{equation}
\int\underline{s}\d\nu\leq\underline{h}(\nu)\leq\overline{h}(\nu)\leq\int\overline{s}\d\nu\label{eq:local-arithmetic-bounds}
\end{equation}
and 
\begin{equation}
\left(\int\frac{1}{\underline{s}}\d\nu\right)^{-1}\leq\underline{D}_{0}(\nu)\leq\overline{D}_{0}(\nu)\leq\left(\int\frac{1}{\overline{s}}\d\nu\right)^{-1}.\label{eq:local-harmonic-bounds}
\end{equation}
All reciprocals and integrals in these formulas are finite and positive.
The bounds \eqref{eq:local-arithmetic-bounds} hold verbatim for the
threshold means $\mathds{h}_{\nu}(t)$ of \eqref{eq:threshold-means}:
\begin{equation}
\int\underline{s}\d\nu\leq\liminf_{t\downarrow0}\mathds{h}_{\nu}(t)\leq\limsup_{t\downarrow0}\mathds{h}_{\nu}(t)\leq\int\overline{s}\d\nu.\label{eq:threshold-arithmetic-bounds}
\end{equation}
\end{proposition}

\begin{proof}
Fatou's lemma and uniform integrability from Lemma~\ref{lem:information-UI}
give \eqref{eq:local-arithmetic-bounds}. For the upper inequality
one can first truncate $X_{k}$ at a constant $w$, apply reverse
Fatou to the bounded truncations, and then let $w\to\infty$; the
tails are uniformly negligible.

By \eqref{eq:first-crossing}, $\liminf_{t\downarrow0}\varsigma_{t}=\underline{s}$
and $\limsup_{t\downarrow0}\varsigma_{t}=\overline{s}$ pointwise.
The functions $1/\varsigma_{t}=\tau_{t}/\log(1/t)$ are uniformly
bounded for $0<t\leq1/2$ by \eqref{eq:max-stopping-depth}, and $\int\varsigma^{-1}_{t}\d\nu=\mathcal{L}_{\nu}(t)/\log(1/t)$
by \eqref{eq:depth-as-integral}. Fatou and reverse Fatou, applied
to $1/\varsigma_{t}$ along arbitrary sequences $t_{n}\downarrow0$,
imply 
\[
\int\frac{1}{\overline{s}}\d\nu\leq\liminf_{t\downarrow0}\frac{\mathcal{L}_{\nu}(t)}{\log(1/t)}\leq\limsup_{t\downarrow0}\frac{\mathcal{L}_{\nu}(t)}{\log(1/t)}\leq\int\frac{1}{\underline{s}}\d\nu.
\]
Reciprocation and Theorem~\ref{thm:adaptive-entropy} prove \eqref{eq:local-harmonic-bounds}.
For \eqref{eq:threshold-arithmetic-bounds}, the second moments of
$\varsigma_{t}$ are bounded for $t\leq1/2$ by Proposition~\ref{prop:threshold-variance}(1),
so Fatou's lemma and reverse Fatou apply to $\varsigma_{t}$ itself,
as for \eqref{eq:local-arithmetic-bounds}.
\end{proof}

\begin{remark}[Attribution]\label{rem:attribution-local}The arithmetic
bounds \eqref{eq:local-arithmetic-bounds} are essentially known.
The upper bound is \cite[Theorem~IV.3]{ST12}, stated there for Rényi's
entropy dimension defined through balls, which coincides with the
dyadic one because the two entropies differ at each scale by an additive
constant depending only on $d$; the essential-supremum version is
\cite[Theorem~1.3]{FLR02}, in the tradition of Young's dimension
comparisons \cite{Young82}. The lower bound is immediate from Fatou's
lemma in the dyadic setting and transfers to balls through Lemma~\ref{lem:balls-restrictions},
a case stated as open in \cite[Section~IV.B]{ST12}. The harmonic
bounds \eqref{eq:local-harmonic-bounds} appear to be new in this
form; Remark~\ref{prop:recovery-zhu} shows how they can also be
obtained from \cite{GL04,Zhu12}.\end{remark}

\begin{remark}[Classical dimension bounds]\label{rem:heurteaux-dimensions}
For the dimensions defined in \eqref{eq:measure-H-definitions}, \cite[Theorems~2.3 and~2.5]{Heu07}
and Lemma~\ref{lem:balls-restrictions} give 
\begin{equation}
\begin{aligned}\dim_{*}(\nu) & =\mathop{\mathrm{ess\,inf}}_{\nu}\underline{s}, & \quad\dim^{*}(\nu) & =\mathop{\mathrm{ess\,sup}}_{\nu}\underline{s},\\
\mathrm{Dim}_{*}(\nu) & =\mathop{\mathrm{ess\,inf}}_{\nu}\overline{s}, & \quad\mathrm{Dim}^{*}(\nu) & =\mathop{\mathrm{ess\,sup}}_{\nu}\overline{s}.
\end{aligned}
\label{eq:measure-dimensions-local}
\end{equation}
Both Hausdorff dimensions use the lower local exponent, and both packing
dimensions use the upper one. The quantization chain of \cite[Theorem~2.1, (2.6)]{Zhu12}
can also be recovered independently from Proposition~\ref{prop:local-liminf-bounds}.
Indeed, the standing mass bound gives $0<a\leq\mathop{\mathrm{ess\,inf}}_{\nu}\underline{s}\leq\mathop{\mathrm{ess\,sup}}_{\nu}\overline{s}\leq d$,
and the harmonic integral bounds imply
\[
\begin{aligned}\mathop{\mathrm{ess\,inf}}_{\nu}\underline{s} & \leq\left(\int\frac{1}{\underline{s}}\d\nu\right)^{-1}\leq\underline{D}_{0}(\nu)\leq\overline{D}_{0}(\nu)\\
 & \leq\left(\int\frac{1}{\overline{s}}\d\nu\right)^{-1}\leq\mathop{\mathrm{ess\,sup}}_{\nu}\overline{s},
\end{aligned}
\]
whose outer terms are $\dim_{*}(\nu)$ and $\mathrm{Dim}^{*}(\nu)$
by \eqref{eq:measure-dimensions-local}. This proves the same classical
bound without using Zhu's theorem: Proposition~\ref{prop:local-liminf-bounds}
follows from Theorem~\ref{thm:adaptive-entropy} and Fatou's lemma.
No almost-everywhere local limit is needed.

\end{remark}

\section{Spectral bounds and neighbouring orders}

\label{sec:entropy}

We compare the general dimension bounds with the ordinary upper spectrum.
The arguments use only upper limits; convergence of the finite-generation
spectra is not assumed.

\subsection{Entropy and one-sided slopes}

\begin{proposition}[Entropy and one-sided slopes]\label{prop:entropy-only}
For $0<q<1<q'$, 
\begin{equation}
\frac{\overline{\beta}(q')}{1-q'}\leq\underline{h}(\nu)\leq\overline{h}(\nu)\leq\frac{\overline{\beta}(q)}{1-q}.\label{eq:entropy-limsup-spectrum-bound}
\end{equation}
Consequently, with the one-sided derivatives of the upper spectrum
from \eqref{eq:reflected-subdifferential}, 
\begin{equation}
\sigma_{-}=-\overline{\beta}'_{+}(1)\leq\underline{h}(\nu)\leq\overline{h}(\nu)\leq-\overline{\beta}'_{-}(1)=\sigma_{+}.\label{eq:one-sided-entropy}
\end{equation}

\end{proposition}
\begin{proof}
Convexity and $\beta_{k}(1)=0$ give 
\begin{equation}
\beta_{k}'(1)=\frac{\sum_{Q\in\D_{k}}\nu(Q)\log\nu(Q)}{k}=-h_{k}(\nu),\label{eq:beta-derivative}
\end{equation}
\[
\frac{\beta_{k}(q')}{1-q'}\leq h_{k}(\nu)\leq\frac{\beta_{k}(q)}{1-q}.
\]
Taking lower and upper limits proves \eqref{eq:entropy-limsup-spectrum-bound};
the denominator $1-q'<0$ turns the liminf on the left into the limsup
of $\beta_{k}(q')$. The function $\overline{\beta}$ is finite and
convex on $(0,\infty)$: pass to the limsup in the convexity inequality.
Letting $q'\downarrow1$ and $q\uparrow1$ proves \eqref{eq:one-sided-entropy}.
This is the entropy part of the classical comparison in \cite[Theorem~3.1]{Heu07}.
\end{proof}

\subsection{Critical parameters and the full spectral enclosure}

The critical values $q_{r}$ of the partition functions of \cite{KNZ,KNneg}
can be located on the ordinary upper spectrum. For $r\geq0$ the supremum
in the set function $\mathfrak{J}_{\nu,r}$ of \cite{KNneg} is attained
at the cube itself and $q_{r}$ solves $\overline{\beta}(q_{r})=rq_{r}$
\cite[Section~1.3]{KNneg}; for negative $r$ the partition function
need not coincide with $\overline{\beta}(q)-rq$, and \cite[Proposition~3.8]{KNhigh}
gives only the two-sided bound $\overline{\beta}(q)-rq\leq\Theta_{r}(q)\leq\overline{\beta}\bigl(q(1+r/\dim_{\infty}(\nu))\bigr)$
in the notation below. The lemma shows that the zero is nevertheless
determined by $\overline{\beta}$.

\begin{lemma}[The critical zero and the ordinary upper spectrum]\label{lem:critical-zero}Assume
\eqref{eq:uniform-dyadic-mass} and let $r>-\dim_{\infty}(\nu)$.
Put
\begin{align*}
\mathcal{J}_{r}(Q) & \coloneqq\sup_{Q'\subseteq Q,\ Q'\in\D}\nu(Q')\ell(Q')^{r},\\
\Theta_{r}(q) & \coloneqq\limsup_{k\to\infty}\frac{1}{k}\log\sum_{Q\in\D_{k}}\mathcal{J}_{r}(Q)^{q},\qquad q>0,
\end{align*}
and $q_{r}\coloneqq\inf\{q>0:\Theta_{r}(q)<0\}$. Then
\begin{equation}
q_{r}\quad\text{is the unique solution of}\quad\overline{\beta}(q_{r})=rq_{r}.\label{eq:critical-root}
\end{equation}
In particular $q_{0}=1$, $q_{r}>1$ for $r<0$, $q_{r}<1$ for $r>0$,
and $q_{r}\to1$ as $r\to0$.\end{lemma}

Up to the side-length normalisation $\ell(Q)^{r}=\Lambda(Q)^{r/d}$,
with $\Lambda$ the Lebesgue measure, $\mathcal{J}_{r}$ is the set
function $\mathfrak{J}_{\nu,r}$ of \cite{KNneg}, $\Theta_{r}$ is
its partition function in the sense of \cite[(1.6)]{KN25}, and $q_{r}$
is the critical value in the quantization formulas of \cite{KNZ,KNneg}.
\begin{proof}
Choose $a_{0}$ with $\max\{0,-r\}<a_{0}<\dim_{\infty}(\nu)$, so
that $\nu(Q)\leq C_{0}2^{-a_{0}k}$ for $Q\in\D_{k}$ and some $C_{0}\geq1$.
Then $\mathcal{J}_{r}(Q)\leq C_{0}2^{-(a_{0}+r)k}$ for $Q\in\D_{k}$,
and on the cubes of positive mass $\mathcal{J}_{r}$ is monotone,
uniformly vanishing and locally non-vanishing with $\infty$-dimension
at least $a_{0}+r>0$ in the sense of \cite{KN25}. By \cite[Lemma~2.4]{KN25},
$q_{r}$ therefore coincides with the critical exponent
\[
\kappa_{r}\coloneqq\inf\Bigl\{ q>0:\sum_{Q\in\D}\mathcal{J}_{r}(Q)^{q}<\infty\Bigr\}.
\]
Write $S_{k}(q)\coloneqq\sum_{Q\in\D_{k}}\nu(Q)^{q}=2^{k\beta_{k}(q)}$.
Since $Q$ itself is admissible in the supremum, and since every $Q'\in\D_{k}$
lies below exactly one cube of each generation $n\leq k$, so that
$\mathcal{J}_{r}(Q)^{q}\leq\sum_{Q'\subseteq Q}\nu(Q')^{q}\ell(Q')^{rq}$
for $q>0$,
\[
\sum_{k\geq0}2^{-rkq}S_{k}(q)\leq\sum_{Q\in\D}\mathcal{J}_{r}(Q)^{q}\leq\sum_{k\geq0}(k+1)2^{-rkq}S_{k}(q).
\]
The $k$th terms of the outer series are $2^{k(\beta_{k}(q)-rq)}$
and $(k+1)2^{k(\beta_{k}(q)-rq)}$, so by the root test both series
converge when $f_{r}(q)\coloneqq\overline{\beta}(q)-rq<0$ and both
diverge when $f_{r}(q)>0$. Hence $\kappa_{r}=\inf\{q>0:f_{r}(q)<0\}$.
For $q_{2}>q_{1}>0$ the mass bound gives $\nu(Q)^{q_{2}}\leq(C_{0}2^{-a_{0}k})^{q_{2}-q_{1}}\nu(Q)^{q_{1}}$,
hence $\overline{\beta}(q_{2})-\overline{\beta}(q_{1})\leq-a_{0}(q_{2}-q_{1})$,
so $f_{r}$ is continuous and strictly decreasing with $f_{r}(q)\to-\infty$;
and $f_{r}(q)\geq a_{0}(1-q)-rq>0$ for small $q$, since $\overline{\beta}(q)\geq a_{0}(1-q)$
for $q<1$. Hence $\kappa_{r}$ is the unique zero of $f_{r}$, which
is \eqref{eq:critical-root}. Since $f_{r}(1)=-r$, we get $q_{0}=1$
and $q_{r}\gtrless1$ for $r\lessgtr0$, and the slope bound gives
$a_{0}|q_{r}-1|\leq|\overline{\beta}(q_{r})|=|r|q_{r}$, hence $q_{r}\to1$
as $r\to0$.
\end{proof}

\begin{theorem}[Spectral enclosure and neighbouring-order limits]\label{thm:spectral-enclosure}
Assume \eqref{eq:uniform-dyadic-mass}, and let $\sigma_{\pm}$ and
$J_{\nu}$ be as in \eqref{eq:reflected-subdifferential}. Then $0<a\leq\sigma_{-}\leq\sigma_{+}\leq d$.
The dyadic local dimensions \eqref{eq:local-exponents} satisfy 
\begin{equation}
\sigma_{-}\leq\underline{s}(x)\leq\overline{s}(x)\leq\sigma_{+}\label{eq:spectral-local-enclosure}
\end{equation}
for $\nu$-almost every $x$. In particular, all four Hausdorff and
packing dimensions of $\nu$ belong to $J_{\nu}$, and 
\begin{align}
\sigma_{-} & \leq\dim_{*}(\nu)\leq\underline{h}(\nu)\leq\overline{h}(\nu)\leq\mathrm{Dim}^{*}(\nu)\leq\sigma_{+},\label{eq:full-entropy-enclosure}\\
\sigma_{-} & \leq\dim_{*}(\nu)\leq\underline{D}_{0}(\nu)\leq\overline{D}_{0}(\nu)\leq\mathrm{Dim}^{*}(\nu)\leq\sigma_{+}.\label{eq:full-quantization-enclosure}
\end{align}
Moreover,
\begin{equation}
\lim_{r\uparrow0}\underline{D}_{r}(\nu)=\sigma_{-},\qquad\lim_{r\downarrow0}\overline{D}_{r}(\nu)=\sigma_{+}.\label{eq:neighbor-limits}
\end{equation}

\end{theorem}
\begin{proof}
For $X_{k}(x)=-\log\nu(Q_{k}(x))/k$, $0<\theta<1$, and $c\in\R$,
Markov's inequality gives 
\begin{align*}
\nu\{X_{k}\leq c\} & \leq2^{k(\beta_{k}(1+\theta)+\theta c)},\\
\nu\{X_{k}\geq c\} & \leq2^{k(\beta_{k}(1-\theta)-\theta c)}.
\end{align*}
Since $\overline{\beta}(1+\theta)=-\theta\sigma_{-}+o(\theta)$ and
$\overline{\beta}(1-\theta)=\theta\sigma_{+}+o(\theta)$, the limsup
of the respective exponential rates is negative for $c=\sigma_{-}-\varepsilon$
and $c=\sigma_{+}+\varepsilon$, if $\theta$ is sufficiently small.
Borel--Cantelli, followed by rational $\varepsilon\downarrow0$,
proves \eqref{eq:spectral-local-enclosure}. This is the probabilistic
argument of \cite[Theorem~3.2]{Heu07}; only a limsup is used. The
elementary bounds for $\sum\nu(Q)^{q}$ using $\max\nu(Q)\leq C2^{-ak}$
and $\#\D_{k}=2^{dk}$ imply $a\leq\sigma_{-}\leq\sigma_{+}\leq d$.
The local characterisations in \cite[Theorems~2.3 and~2.5]{Heu07},
and the dyadic--ball comparison in Lemma~\ref{lem:balls-restrictions},
place all four classical dimensions in $J_{\nu}$. The entropy bounds
are \cite[Theorem~3.1]{Heu07}. The outer bounds of the quantization
chain in \eqref{eq:full-quantization-enclosure} follow without any
appeal to local dimensions: reciprocating the depth bounds \eqref{eq:depth-sigma-bounds}
of Lemma~\ref{lem:stopping-size} in Theorem~\ref{thm:adaptive-entropy}
gives $\sigma_{-}\leq\underline{D}_{0}(\nu)$ and $\overline{D}_{0}(\nu)\leq\sigma_{+}$.
The inner bounds $\dim_{*}(\nu)\leq\underline{D}_{0}(\nu)$ and $\overline{D}_{0}(\nu)\leq\mathrm{Dim}^{*}(\nu)$
are Zhu's \cite[Theorem~2.1, (2.6)]{Zhu12}, whose uniform ball-mass
hypothesis follows from compact support and \eqref{eq:uniform-dyadic-mass};
Remark~\ref{rem:heurteaux-dimensions} recovers them from the stopping-time
integral bounds. This proves \eqref{eq:full-entropy-enclosure}--\eqref{eq:full-quantization-enclosure}.
For \eqref{eq:neighbor-limits}, Lemma~\ref{lem:critical-zero} gives
$\overline{\beta}(q_{r})=rq_{r}$, with $q_{r}>1$ for $r<0$ and
$q_{r}<1$ for $r>0$. The same lemma implies $q_{r}\to1$ as $r\to0$.
The formulas of \cite[Theorem~1.5(i)--(ii)]{KNneg} therefore give
\[
\underline{D}_{r}(\nu)=\frac{\overline{\beta}(q_{r})}{1-q_{r}}\quad(-\dim_{\infty}(\nu)<r<0),\qquad\overline{D}_{r}(\nu)=\frac{\overline{\beta}(q_{r})}{1-q_{r}}\quad(r>0).
\]
The one-sided secants tend to $\sigma_{-}$ and $\sigma_{+}$, respectively.
Finally, power-mean monotonicity yields 
\[
\mathfrak{e}_{N,r_{-}}(\nu)\leq\mathfrak{e}_{N,0}(\nu)\leq\mathfrak{e}_{N,r_{+}}(\nu),\qquad r_{-}<0<r_{+},
\]
within the range of positive finite errors. Consequently 
\[
\underline{D}_{r_{-}}(\nu)\leq\underline{D}_{0}(\nu),\qquad\overline{D}_{0}(\nu)\leq\overline{D}_{r_{+}}(\nu),
\]
which also recovers the outer spectral bounds $\sigma_{-}\leq\underline{D}_{0}(\nu)\leq\overline{D}_{0}(\nu)\leq\sigma_{+}$
directly from \eqref{eq:neighbor-limits}.
\end{proof}

\begin{remark}[Scope of the spectral interval]\label{rem:spectral-interval-scope}
The interval $J_{\nu}$ contains the limiting entropy, order-zero
quantization and classical measure dimensions, and the almost-everywhere
local exponents. It sits inside the outer bounds
\[
\dim_{\infty}(\nu)\leq\sigma_{-},\qquad\sigma_{+}\leq\overline{\beta}(0)\leq d,
\]
the first by admitting every $a<\dim_{\infty}(\nu)$ in \eqref{eq:uniform-dyadic-mass},
the second by the secant between $0$ and $1$. For a biased binary
Bernoulli measure both are strict while $J_{\nu}$ is a singleton.\end{remark}

\section{Local convergence and equality cases}

\label{sec:convergence}

We now impose convergence of the local information rates $X_{k}$
under $\nu$, in distribution or almost everywhere. The main statement
is Proposition~\ref{prop:spectral-window}; the almost-everywhere
case, which yields the arithmetic and harmonic means of the local
dimension, is recorded in Remark~\ref{cor:local-harmonic}, and Theorem~\ref{thm:quantitative-gap}
quantifies the comparison between the two means.

\begin{proposition}[The rescaled spectra at one and convergence in distribution]\label{prop:spectral-window}Assume
\eqref{eq:uniform-dyadic-mass}. For $k\geq1$ and $\theta\in\mathbb{R}$
put
\begin{equation}
\Psi_{k}(\theta)\coloneqq k\ln2\;\beta_{k}\Bigl(1+\frac{\theta}{k\ln2}\Bigr)=\ln\int\e^{-\theta X_{k}}\d\nu,\label{eq:window-definition}
\end{equation}
the cumulant generating function of $-X_{k}$ under $\nu$.\begin{enumerate}\item
$\Psi_{k}$ is finite and convex, and for $n\geq1$
\begin{equation}
\Psi^{(n)}_{k}(0)=\frac{\beta^{(n)}_{k}(1)}{(k\ln2)^{n-1}}=(-1)^{n}\kappa_{n}(X_{k}),\label{eq:window-cumulants}
\end{equation}
where $\kappa_{n}(X_{k})$ is the $n$-th cumulant of $X_{k}$ under
$\nu$. In particular $\Psi_{k}'(0)=-h_{k}(\nu)$ and $\Psi_{k}''(0)=\operatorname{Var}_{\nu}(X_{k})=\beta_{k}''(1)/(k\ln2)$.\item
The laws of $X_{k}$ under $\nu$ converge weakly as $k\to\infty$
if and only if $\Psi_{k}(\theta)$ converges in $\mathbb{R}$ for
every $\theta\in\mathbb{R}$.\item If the laws of $X_{k}$ converge
weakly to a probability measure $\varrho$, then $\varrho$ is carried
by the reflected subdifferential $J_{\nu}=[\sigma_{-},\sigma_{+}]$
of \eqref{eq:reflected-subdifferential}, $\Psi(\theta)\coloneqq\lim_{k}\Psi_{k}(\theta)=\ln\int\e^{-\theta y}\d\varrho(y)$,
all cumulants of $X_{k}$ converge to those of $\varrho$, and both
dimensions exist:
\begin{equation}
h(\nu)=\int y\d\varrho(y)=-\Psi'(0),\qquad D_{0}(\nu)=\Bigl(\int\frac{1}{y}\d\varrho(y)\Bigr)^{-1}=\Bigl(\int^{\infty}_{0}e^{\Psi(\theta)}\d\theta\Bigr)^{-1};\label{eq:window-harmonic}
\end{equation}
moreover $\beta_{k}''(1)/(k\ln2)\to\operatorname{Var}(\varrho)=\Psi''(0)$.\end{enumerate}\end{proposition}
\begin{proof}
For (1), $\sum_{Q\in\D_{k}}\nu(Q)^{q}=\int\nu(Q_{k}(x))^{q-1}\d\nu(x)=\int\e^{-(q-1)k\ln2\,X_{k}}\d\nu$
gives $\beta_{k}(q)=(k\ln2)^{-1}\ln\int e^{-(q-1)k\ln2\,X_{k}}\d\nu$,
and \eqref{eq:window-definition} follows by substituting $q=1+\theta/(k\ln2)$.
The integral is a finite sum over the cubes of positive mass, so $\Psi_{k}$
is finite and smooth; it is convex by Hölder's inequality, and its
derivatives at $0$ are the cumulants of $-X_{k}$. The chain rule
gives the first equality in \eqref{eq:window-cumulants}; for $n=1$
this is \eqref{eq:beta-derivative}.

By Lemma~\ref{lem:information-UI} the laws of the $X_{k}$ are tight
and, for every $\theta$ and $n$, the families $(e^{-\theta X_{k}})_{k}$
and $(X^{n}_{k})_{k}$ are uniformly integrable. If the laws of $X_{k}$
converge weakly to $\varrho$, the portmanteau theorem shows that
$\varrho$ is carried by $[\sigma_{-},\sigma_{+}]$: for $c<\sigma_{-}$
there is $\theta\in(0,1)$ with $\overline{\beta}(1+\theta)+\theta c<0$,
so $\nu\{X_{k}<c\}\to0$ by \eqref{eq:markov-window} and hence $\varrho((-\infty,c))\leq\liminf_{k}\nu\{X_{k}<c\}=0$;
likewise $\varrho((c,\infty))=0$ for $c>\sigma_{+}$. Uniform integrability
gives $\int\e^{-\theta X_{k}}\d\nu\to\int\e^{-\theta y}\d\varrho(y)$
and $\int X^{n}_{k}\d\nu\to\int y^{n}\d\varrho(y)$ for all $\theta$
and $n$. This proves the direct implication in (2), the formula for
$\Psi$ and the convergence of the cumulants in (3), and, by (1),
the statements $h_{k}(\nu)\to\int y\d\varrho$ and $\beta_{k}''(1)/(k\ln2)\to\operatorname{Var}(\varrho)$.
Conversely, if $\Psi_{k}(\theta)\to\Psi(\theta)$ for all $\theta$,
every subsequence of the laws of $X_{k}$ has, by tightness, a further
weakly convergent subsequence with some limit $\varrho'$, and by
what has just been shown $\int\e^{-\theta y}\d\varrho'(y)=\e^{\Psi(\theta)}$
for all $\theta>0$. A finite measure on a half-line is determined
by its Laplace transform, so all subsequential limits coincide, and
the laws of $X_{k}$ converge weakly. This proves (2).

It remains to prove the harmonic formula in (3). Put $W_{k}(c)\coloneqq\nu\{X_{k}\leq c\}$
for $k\geq1$ and $W(c)\coloneqq\varrho((-\infty,c])$. By \eqref{eq:weighted-count}
and its proof, $\mathcal{L}_{\nu}(t)=\sum_{k\geq0}\nu\{I_{k}\leq\ell_{t}\}=1+\sum_{k\geq1}W_{k}(\ell_{t}/k)$,
and writing the sum as an integral over $y\in(k-1,k]$ and substituting
$y=\ell_{t}z$,
\[
\frac{\mathcal{L}_{\nu}(t)}{\ell_{t}}=\frac{1}{\ell_{t}}+\int^{\infty}_{0}W_{\lceil\ell_{t}z\rceil}\Bigl(\frac{\ell_{t}}{\lceil\ell_{t}z\rceil}\Bigr)\d z.
\]
Fix $z>0$ such that $1/z$ is a continuity point of $W$. As $t\downarrow0$,
$k=\lceil\ell_{t}z\rceil\to\infty$ and $\ell_{t}/k\to1/z$, so for
continuity points $c'<1/z<c''$ of $W$ eventually $W_{k}(c')\leq W_{k}(\ell_{t}/k)\leq W_{k}(c'')$,
and weak convergence gives $W(c')\leq\liminf\leq\limsup\leq W(c'')$;
letting $c'\uparrow1/z$ and $c''\downarrow1/z$ shows that the integrand
converges to $W(1/z)$ for all but countably many $z$. Moreover,
fix $a_{0}\in(0,a)$ and $k_{1}$ with $a-(\log C)/k\geq a_{0}$ for
$k\geq k_{1}$; for $z>1/a_{0}$ and $\ell_{t}\geq k_{1}a_{0}$ we
have $\lceil\ell_{t}z\rceil\geq k_{1}$ and $\ell_{t}/\lceil\ell_{t}z\rceil<a_{0}$,
so the integrand vanishes there. Dominated convergence yields
\[
\lim_{t\downarrow0}\frac{\mathcal{L}_{\nu}(t)}{\log(1/t)}=\int^{\infty}_{0}W(1/z)\d z=\int\int^{\infty}_{0}\mathbf{1}_{\{z\leq1/y\}}\d z\d\varrho(y)=\int\frac{1}{y}\d\varrho(y)\in\Bigl[\frac{1}{\sigma_{+}},\frac{1}{\sigma_{-}}\Bigr],
\]
and Theorem~\ref{thm:adaptive-entropy} gives the first expression
for $D_{0}(\nu)$ in \eqref{eq:window-harmonic}. The second follows
from $1/y=\int^{\infty}_{0}e^{-\theta y}\d\theta$ and Tonelli's theorem.
\end{proof}

\begin{remark}[Almost-everywhere convergence]\label{cor:local-harmonic}If
the dyadic local dimension
\[
s(x)\coloneqq\lim_{k\to\infty}X_{k}(x)
\]
exists for $\nu$-almost every $x$, then the laws of $X_{k}$ converge
weakly to $\varrho=s_{*}\nu$, and Proposition~\ref{prop:spectral-window}(3)
gives
\begin{equation}
h(\nu)=\int s(x)\d\nu(x),\qquad D_{0}(\nu)=\Bigl(\int\frac{1}{s(x)}\d\nu(x)\Bigr)^{-1},\label{eq:local-arithmetic-formula}
\end{equation}
the classical arithmetic and harmonic means of the local dimension,
together with $\beta_{k}''(1)/(k\ln2)\to\operatorname{Var}_{\nu}(s)$.
Strict convexity of $z\mapsto1/z$ gives $D_{0}(\nu)\leq h(\nu)$,
with equality precisely when $s$ is almost everywhere constant. A
constant local dimension $s>0$ yields
\[
h(\nu)=D_{0}(\nu)=s.
\]
No decomposition into separated components is involved. Pointwise,
$I_{k}(x)/k\to s(x)$ is equivalent by \eqref{eq:first-crossing}
to
\begin{equation}
\varsigma_{t}(x)\longrightarrow s(x),\label{eq:inverse-local-growth}
\end{equation}
so that the harmonic formula is also the dominated-convergence limit
of Theorem~\ref{thm:adaptive-entropy} applied to $1/\varsigma_{t}=\tau_{t}/\log(1/t)$,
which is bounded by \eqref{eq:max-stopping-depth}. The arithmetic
formula follows from uniform integrability of the information rates;
see Remark~\ref{rem:attribution-local} for the attribution to \cite{FLR02,ST12}.
The harmonic formula can alternatively be assembled from \cite[Example~4.1]{GL04}
and \cite[Lemma~2.2 and Theorem~2.1]{Zhu12}, see Remark~\ref{prop:recovery-zhu};
we have not found it stated in this generality.\end{remark}
\begin{example}[Convergence in law without almost-everywhere convergence]
\label{ex:law-not-ae} Fix $0<p<\tfrac{1}{2}$ and put $s\coloneqq h_{2}(p)=-p\log p-(1-p)\log(1-p)\in(0,1)$.
Let $n_{i}\coloneqq i!$, so that $n_{i+1}/n_{i}\to\infty$, and $c_{i}\coloneqq\lceil\log i\rceil\vee1$,
so that $\varepsilon_{i}\coloneqq2^{-c_{i}}$ satisfies $\varepsilon_{i}\to0$
and $\sum_{i}\varepsilon_{i}=\infty$. Let $\nu$ be the law on $[0,1]$
of a random binary expansion whose digits in the $i$-th epoch $[n_{i},n_{i+1})$
are generated as follows, independently across epochs: for $i\geq3$
the first $c_{i}$ digits are fair (all digits of the first two epochs
are fair), and the epoch is called \emph{exceptional} if they all
vanish, an event of probability $\varepsilon_{i}$; the remaining
digits of the epoch are independent Bernoulli$(p)$ digits if the
epoch is not exceptional and fair digits if it is. Every conditional
digit probability is at most $1-p$, so \eqref{eq:uniform-dyadic-mass}
holds with $a=-\log(1-p)$ and $C=1$.
\end{example}

Write $\rho_{i}\coloneqq1$ if epoch $i$ is exceptional and $\rho_{i}\coloneqq s$
otherwise. Conditional on all the initial coin words, the remaining
information increments are independent and uniformly bounded, with
mean $\rho_{i}$ in epoch $i$. The coin-word information increments
themselves equal one. Thus the conditional strong law and $\sum_{l\leq i}c_{l}+n_{i-1}=o(n_{i})$
give, $\nu$-almost surely,
\begin{equation}
X_{k}=\frac{n_{i}}{k}\,\rho_{i-1}+\frac{k-n_{i}}{k}\,\rho_{i}+o(1)\qquad(n_{i}\leq k<n_{i+1}).\label{eq:regime-average}
\end{equation}
Since $\nu\{\rho_{i}=1\text{ or }\rho_{i-1}=1\}\leq\varepsilon_{i}+\varepsilon_{i-1}\to0$,
\eqref{eq:regime-average} gives $X_{k}\to s$ in $\nu$-probability:
the laws of $X_{k}$ converge weakly to $\varrho=\delta_{s}$. On
the other hand the events $\{\rho_{i}=1\}$ are independent with $\sum_{i}\varepsilon_{i}=\infty$,
so by the second Borel--Cantelli lemma almost every $x$ lies in
infinitely many exceptional and infinitely many non-exceptional epochs;
evaluating \eqref{eq:regime-average} at $k=n_{i+1}-1$, where $n_{i}/k\to0$,
yields $\underline{s}(x)=s<1=\overline{s}(x)$ for $\nu$-almost every
$x$. The dyadic local dimension therefore exists almost nowhere,
while Proposition~\ref{prop:spectral-window} applies and gives $h(\nu)=D_{0}(\nu)=s$
and $\beta_{k}''(1)=o(k)$. By \eqref{eq:measure-dimensions-local},
$\dim_{*}(\nu)=\dim^{*}(\nu)=s$ and $\mathrm{Dim}_{*}(\nu)=\mathrm{Dim}^{*}(\nu)=1$,
so the integral bounds \eqref{eq:local-harmonic-bounds} only locate
$D_{0}(\nu)$ in $[s,1]$; the spectral window pins it to the left
endpoint.

The spectrum is explicit. Summing $\nu(Q)^{q}$ over the digit strings
of one complete epoch gives a factor $(2^{c_{l}}-1)2^{-c_{l}q}A(q)^{n'_{l}}+2^{-c_{l}q}2^{(1-q)n'_{l}}$
with $A(q)\coloneqq p^{q}+(1-p)^{q}$ and $n'_{l}\coloneqq n_{l+1}-n_{l}-c_{l}$;
by strict convexity for orders above one and strict concavity for
orders between zero and one, $A(q)>2^{1-q}$ for $q>1$ and $A(q)<2^{1-q}$
for $q<1$, so the larger of the two terms determines the exponential
rate and the coin digits contribute $O(\sum_{l\leq i}c_{l})=o(k)$.
Hence the limit exists and
\begin{equation}
\beta(q)=\max\bigl\{1-q,\ \log\bigl(p^{q}+(1-p)^{q}\bigr)\bigr\}\qquad(q>0),\label{eq:law-not-ae-spectrum}
\end{equation}
with a corner at $q=1$: $\sigma_{-}=h_{2}(p)=s$ and $\sigma_{+}=1$.
Thus $\varrho=\delta_{\sigma_{-}}$, the two means sit at the left
endpoint of $J_{\nu}=[s,1]$, and the packing dimensions at the right
one.

The same construction with a summable sequence, for instance $c_{i}\coloneqq\lceil2\log i\rceil\vee1$
and hence $\varepsilon_{i}\leq i^{-2}$, changes only the last step:
by the first Borel--Cantelli lemma almost every $x$ lies in finitely
many exceptional epochs, so \eqref{eq:regime-average} gives $X_{k}\to s$
almost everywhere and the local dimension is the constant $s$. The
spectrum \eqref{eq:law-not-ae-spectrum} is unchanged, since only
$\sum_{l\leq i}c_{l}=o(n_{i})$ entered its computation. Thus $h(\nu)=D_{0}(\nu)=s$
and all four classical dimensions equal $s$, while $\overline{\beta}$
has a first-order phase transition at $q=1$ with $J_{\nu}=[s,1]$:
equality of the two means does not force differentiability of the
spectrum at $1$, even for an almost everywhere constant local dimension,
and $\beta_{k}''(1)=o(k)$ although the limiting corner does not close.
The two variants differ only in whether the exceptional epochs recur.

\begin{theorem}[Sharp comparison, defect, and spectral curvature]\label{thm:quantitative-gap}
Assume \eqref{eq:uniform-dyadic-mass} and that the laws of $X_{k}$
under $\nu$ converge weakly to $\varrho$, as in Proposition~\ref{prop:spectral-window}.
Let $m_{\varrho}\coloneqq\min\operatorname{supp}\varrho$ and $M_{\varrho}\coloneqq\max\operatorname{supp}\varrho$,
and write $h=h(\nu)$ and $D=D_{0}(\nu)$. Then $0<\sigma_{-}\leq m_{\varrho}\leq M_{\varrho}\leq\sigma_{+}$
and
\begin{equation}
\frac{m_{\varrho}M_{\varrho}}{m_{\varrho}+M_{\varrho}-h}\leq D\leq h,\qquad0\leq h-D\leq(\sqrt{M_{\varrho}}-\sqrt{m_{\varrho}})^{2}.\label{eq:sharp-gap-bounds}
\end{equation}
The defect has the exact representations
\begin{equation}
h-D=\int\frac{(y-D)^{2}}{y}\d\varrho(y)=\frac{D}{h}\int\frac{(y-h)^{2}}{y}\d\varrho(y).\label{eq:exact-gap-identity}
\end{equation}
In particular,
\begin{equation}
\frac{D}{hM_{\varrho}}\operatorname{Var}(\varrho)\leq h-D\leq\frac{D}{hm_{\varrho}}\operatorname{Var}(\varrho),\qquad\lim_{k\to\infty}\frac{\beta_{k}''(1)}{k\ln2}=\operatorname{Var}(\varrho).\label{eq:gap-curvature}
\end{equation}
Hence $h=D$ if and only if $\varrho=\delta_{h}$, equivalently $\beta_{k}''(1)=o(k)$.
If $m_{\varrho}<M_{\varrho}$, equality in the lower bound for $D$
holds exactly when $\varrho(\{m_{\varrho},M_{\varrho}\})=1$, and
the maximal gap $(\sqrt{M_{\varrho}}-\sqrt{m_{\varrho}})^{2}$ occurs
exactly for
\begin{equation}
\varrho(\{m_{\varrho}\})=\frac{\sqrt{m_{\varrho}}}{\sqrt{m_{\varrho}}+\sqrt{M_{\varrho}}},\qquad\varrho(\{M_{\varrho}\})=\frac{\sqrt{M_{\varrho}}}{\sqrt{m_{\varrho}}+\sqrt{M_{\varrho}}}.\label{eq:maximal-gap-law}
\end{equation}
For a two-point law $\varrho=p\delta_{s_{1}}+(1-p)\delta_{s_{2}}$
the first representation in \eqref{eq:exact-gap-identity} reads
\begin{equation}
h-D=\frac{p(1-p)(s_{1}-s_{2})^{2}}{ps_{2}+(1-p)s_{1}}.\label{eq:two-component-gap}
\end{equation}
The bounds \eqref{eq:sharp-gap-bounds} and \eqref{eq:gap-curvature}
remain valid with any interval $[m_{\varrho},M_{\varrho}]\subseteq[\alpha,\alpha']\subset(0,\infty)$
in place of $[m_{\varrho},M_{\varrho}]$; in particular,
\begin{equation}
h(\nu)-D_{0}(\nu)\leq(\sqrt{M_{\varrho}}-\sqrt{m_{\varrho}})^{2}\leq(\sqrt{\sigma_{+}}-\sqrt{\sigma_{-}})^{2}.\label{eq:spectral-gap-bound}
\end{equation}
If $X_{k}\to s$ almost everywhere, then $\varrho=s_{*}\nu$ and $[m_{\varrho},M_{\varrho}]=[\dim_{*}(\nu),\mathrm{Dim}^{*}(\nu)]$,
so that
\begin{equation}
h(\nu)-D_{0}(\nu)\leq\bigl(\sqrt{\mathrm{Dim}^{*}(\nu)}-\sqrt{\dim_{*}(\nu)}\bigr)^{2}\leq(\sqrt{\sigma_{+}}-\sqrt{\sigma_{-}})^{2}.\label{eq:classical-gap-bound}
\end{equation}
\end{theorem}
\begin{proof}
By Proposition~\ref{prop:spectral-window}(3), $\varrho$ is carried
by $[\sigma_{-},\sigma_{+}]$, so $\sigma_{-}\leq m_{\varrho}\leq M_{\varrho}\leq\sigma_{+}$,
and $\sigma_{-}\geq a>0$ by Lemma~\ref{lem:stopping-size}. By the
same proposition, $h=\int y\d\varrho(y)$, $D^{-1}=\int y^{-1}\d\varrho(y)$
and $\beta_{k}''(1)/(k\ln2)\to\operatorname{Var}(\varrho)$; both
integrals are finite because $\varrho$ is carried by $[m_{\varrho},M_{\varrho}]$
with $m_{\varrho}>0$. Since $\varrho$ is a probability measure,
expanding the two integrands of \eqref{eq:exact-gap-identity} gives
\[
\int\frac{(y-D)^{2}}{y}\d\varrho=\int y\d\varrho-2D+D^{2}\int\frac{d\varrho}{y}=h-2D+D=h-D
\]
and
\[
\int\frac{(y-h)^{2}}{y}\d\varrho=\int y\d\varrho-2h+h^{2}\int\frac{d\varrho}{y}=h-2h+\frac{h^{2}}{D}=\frac{h}{D}(h-D),
\]
which is \eqref{eq:exact-gap-identity}. Bounding $1/y$ between $1/M_{\varrho}$
and $1/m_{\varrho}$ in the second representation gives the variance
bounds in \eqref{eq:gap-curvature}. For $m_{\varrho}<M_{\varrho}$,
the classical convex secant bound (cf. Edmundson \cite{Edm57}) for
the strictly convex function $z\mapsto1/z$ is
\[
\frac{1}{z}\leq\frac{m_{\varrho}+M_{\varrho}-z}{m_{\varrho}M_{\varrho}}\qquad(m_{\varrho}\leq z\leq M_{\varrho}),
\]
with equality exactly at the endpoints. Integrating against $\varrho$
gives $D\geq m_{\varrho}M_{\varrho}/(m_{\varrho}+M_{\varrho}-h)$.
Maximising $h-m_{\varrho}M_{\varrho}/(m_{\varrho}+M_{\varrho}-h)$
over $h\in[m_{\varrho},M_{\varrho}]$ gives its unique maximum $(\sqrt{M_{\varrho}}-\sqrt{m_{\varrho}})^{2}$
at $h=m_{\varrho}+M_{\varrho}-\sqrt{m_{\varrho}M_{\varrho}}$, and
equality forces the endpoint law \eqref{eq:maximal-gap-law}, for
which $D=\sqrt{m_{\varrho}M_{\varrho}}$. If $m_{\varrho}=M_{\varrho}$,
all assertions reduce to $h=D=m_{\varrho}$. The two-point formula
\eqref{eq:two-component-gap} is the first representation evaluated
at $\varrho=p\delta_{s_{1}}+(1-p)\delta_{s_{2}}$. The lower bound
$m_{\varrho}M_{\varrho}/(m_{\varrho}+M_{\varrho}-h)$ is non-decreasing
in $m_{\varrho}$ and non-increasing in $M_{\varrho}$ for $h\in[m_{\varrho},M_{\varrho}]$,
and $(\sqrt{M_{\varrho}}-\sqrt{m_{\varrho}})^{2}$, $D/(hM_{\varrho})$
and $D/(hm_{\varrho})$ are monotone in the same directions, which
gives the passage to a larger interval and, with $[\sigma_{-},\sigma_{+}]$,
proves \eqref{eq:spectral-gap-bound}. If $X_{k}\to s$ almost everywhere,
the laws of $X_{k}$ converge to $\varrho=s_{*}\nu$, so $m_{\varrho}=\mathop{\mathrm{ess\,inf}}_{\nu}s$
and $M_{\varrho}=\mathop{\mathrm{ess\,sup}}_{\nu}s$, which equal
$\dim_{*}(\nu)$ and $\mathrm{Dim}^{*}(\nu)$ by \eqref{eq:measure-dimensions-local};
this turns \eqref{eq:spectral-gap-bound} into \eqref{eq:classical-gap-bound}.
Sharpness follows from Corollary~\ref{cor:prescribed-local-law},
which realises each of the endpoint laws by a one-dimensional Bernoulli
mixture.
\end{proof}

The second inequality in \eqref{eq:spectral-gap-bound} can be strict:
in Example~\ref{ex:law-not-ae} one has $\varrho=\delta_{s}$ and
hence $m_{\varrho}=M_{\varrho}=s$, while $[\sigma_{-},\sigma_{+}]=[s,1]$.
In \eqref{eq:classical-gap-bound} the inner bound is sharp and the
outer one is computable from the spectrum alone; the bounds in \eqref{eq:sharp-gap-bounds}
are attained by one-dimensional measures with prescribed local-dimension
laws in any interval $0<m_{\varrho}<M_{\varrho}\leq1$ (Corollary~\ref{cor:prescribed-local-law}).

\begin{remark}[Classical dimensions under local convergence]\label{rem:classical-local-convergence}
If the local limit $s$ exists, then $\varrho=s_{*}\nu$, $m_{\varrho}=\mathop{\mathrm{ess\,inf}}_{\nu}s$
and $M_{\varrho}=\mathop{\mathrm{ess\,sup}}_{\nu}s$, and
\[
\sigma_{-}\leq m_{\varrho}=\dim_{*}(\nu)=\mathrm{Dim}_{*}(\nu)\leq D_{0}(\nu)\leq h(\nu)\leq\dim^{*}(\nu)=\mathrm{Dim}^{*}(\nu)=M_{\varrho}\leq\sigma_{+},
\]
with $m_{\varrho}<D_{0}(\nu)<h(\nu)<M_{\varrho}$ strict for non-constant
$s$. Writing $\nu_{E}=\nu|_{E}/\nu(E)$,
\[
\inf_{\nu(E)>0}h(\nu_{E})=\inf_{\nu(E)>0}D_{0}(\nu_{E})=m_{\varrho},\qquad\sup_{\nu(E)>0}h(\nu_{E})=\sup_{\nu(E)>0}D_{0}(\nu_{E})=M_{\varrho},
\]
since normalised restrictions inherit the mass bound and the local
exponent, so that their dimensions are the conditional arithmetic
and harmonic means of $s$; restricting to $\{s<m_{\varrho}+\varepsilon\}$
or $\{s>M_{\varrho}-\varepsilon\}$ gives the endpoints. Thus the
monotone quantization envelopes of \cite[Section~4, (4.3)]{GL04}
equal the upper Hausdorff and packing dimensions in the local-limit
setting.\end{remark}

\subsection{Equality criteria and regular measures}

\begin{corollary}[A differentiable upper spectrum gives equality]\label{cor:KN-differentiable}Assume
\eqref{eq:uniform-dyadic-mass}. If $\overline{\beta}$ is differentiable
at $1$, then the local dimension exists and equals $s=-\overline{\beta}'(1)$
almost everywhere. In particular, 
\[
\underline{h}(\nu)=\overline{h}(\nu)=\underline{D}_{0}(\nu)=\overline{D}_{0}(\nu)=s.
\]
All four Hausdorff and packing dimensions have the same value.

\end{corollary}
\begin{proof}
The interval in \eqref{eq:spectral-local-enclosure} reduces to a
singleton. Remark~\ref{cor:local-harmonic} and \eqref{eq:measure-dimensions-local}
therefore prove all assertions. The dimension identity is also a direct
consequence of Kesseböhmer--Niemann \cite[Theorem~1.5(iii)]{KNneg}:
Lemma~\ref{lem:critical-zero} gives $r=\overline{\beta}(q_{r})/q_{r}$
and 
\[
q_{r}\longrightarrow1,\qquad\left.\frac{dq_{r}}{dr}\right|_{r=0}=\frac{1}{\overline{\beta}'(1)}.
\]
Thus their differentiability hypothesis at order zero holds, and $D_{0}(\nu)=-\overline{\beta}'(1)$.
Proposition~\ref{prop:entropy-only} gives the same entropy value
independently of Theorem~\ref{thm:adaptive-entropy}. As noted in
\cite{KNneg}, the order-zero statement there is itself obtained by
combining \cite{Zhu12} with \cite{Heu07}; the corollary is therefore
not new, but it identifies the differentiable case as the degenerate
instance of the enclosure \eqref{eq:spectral-local-enclosure}.
\end{proof}

\begin{remark}[The harmonic formula from Graf–Luschgy and Zhu]\label{prop:recovery-zhu}Zhu's
Theorem 2.1 gives $\mathop{\mathrm{ess\,inf}}_{\mu}\underline{d}_{\mu}\leq\underline{D}_{0}(\mu)\leq\overline{D}_{0}(\mu)\leq\mathop{\mathrm{ess\,sup}}_{\mu}\overline{d}_{\mu}$
for every compactly supported $\mu$ satisfying \eqref{eq:uniform-ball-mass},
with $\underline{d}_{\mu},\overline{d}_{\mu}$ the ball local dimensions,
and his Lemma 2.2 gives the harmonic bounds $(\sum_{j}p_{j}/\underline{D}_{0}(\mu_{j}))^{-1}\leq\underline{D}_{0}(\mu)\leq\overline{D}_{0}(\mu)\leq(\sum_{j}p_{j}/\overline{D}_{0}(\mu_{j}))^{-1}$
for two-component mixtures $\mu=\sum_{j}p_{j}\mu_{j}$, hence by induction
for finite ones, compare \cite[Example~4.1]{GL04} and \cite[Proposition~3.8]{Zhu15}.
Partitioning $[a,d]$ into disjoint half-open intervals (with the
last endpoint included), whose closures are $[a_{j},b_{j}]$ of length
at most $\delta$ and restricting $\nu$ to $E_{j}=\{x:s(x)\text{ lies in the }j\text{-th partition interval}\}$,
which by Lemma~\ref{lem:balls-restrictions} preserves the mass bound
and the local exponents, squeezes $\underline{D}_{0}(\nu)$ and $\overline{D}_{0}(\nu)$
between $(\sum_{j}p_{j}/a_{j})^{-1}$ and $(\sum_{j}p_{j}/b_{j})^{-1}$
with $p_{j}=\nu(E_{j})$; letting $\delta\downarrow0$ gives the harmonic
formula in \eqref{eq:local-arithmetic-formula}, and the same argument
for a partition of the range of $(\underline{s},\overline{s})$ gives
\eqref{eq:local-harmonic-bounds}. This argument uses the extremal
bounds and finite-mixture inequalities directly.\end{remark}

\section{Prescribed local dimensions via Bernoulli mixtures}

\label{sec:bernoulli-mixtures}

We now construct one-dimensional measures whose dyadic local dimension
exists almost everywhere but has a continuous, rather than finitely
supported, distribution. All these measures have full support $[0,1]$
and satisfy the standing hypothesis \eqref{eq:uniform-dyadic-mass}.
In particular, geometric separation of the different local-dimension
levels is not needed.

\subsection{A random-parameter construction}

For $p\in(0,1)$, let $\mu_{p}$ be the distribution on $[0,1]$ of
\[
X=\sum^{\infty}_{j=1}\frac{\omega_{j}}{2^{j}},
\]
where the digits $\omega_{j}$ are independent and satisfy $\mathbb{P}(\omega_{j}=1)=p$.
If $Q\in\D_{n}$ is represented by a binary word containing $k$ ones,
then 
\begin{equation}
\mu_{p}(Q)=p^{k}(1-p)^{n-k}.\label{eq:Bernoulli-cylinder}
\end{equation}
 Dyadic endpoints have zero mass and can be disregarded when identifying
a point with its binary expansion.

Fix $0<p_{-}<p_{+}\leq1/2$ and a Borel probability measure $\rho$
supported on $[p_{-},p_{+}]$. Define 
\begin{equation}
\nu_{\rho}(E)\coloneqq\int^{p_{+}}_{p_{-}}\mu_{p}(E)\d\rho(p)\label{eq:bernoulli-mixture}
\end{equation}
for Borel sets $E\subseteq[0,1]$. Equivalently, first choose a single
parameter $P$ with distribution $\rho$ and, conditional on $P=p$,
generate all digits independently with success probability $p$. Thus
\begin{equation}
\nu_{\rho}(Q)=\int^{p_{+}}_{p_{-}}p^{k}(1-p)^{n-k}\d\rho(p).\label{eq:bernoulli-mixture-cylinder}
\end{equation}
The parameter is chosen once and then kept fixed; choosing an independent
parameter for every digit would instead produce one ordinary Bernoulli
measure. This random-parameter construction is the Bernoulli-mixture
representation associated with de Finetti's theorem; see \cite{Kirsch}.
We only use the explicit construction and prove its local-dimension
properties directly below.

\begin{proposition}[Local dimensions of Bernoulli mixtures]\label{prop:bernoulli-local-law}
Let $\nu_{\rho}$ be as in \eqref{eq:bernoulli-mixture}, and set
$h_{2}(p)\coloneqq-p\log p-(1-p)\log(1-p).$ Then $\nu_{\rho}$ is
atomless, has support $[0,1]$, and satisfies 
\begin{equation}
\nu_{\rho}(Q)\leq(1-p_{-})^{n}\quad(Q\in\D_{n}),\label{eq:bernoulli-mass-exponent}
\end{equation}
so that \eqref{eq:uniform-dyadic-mass} holds with $a=-\log(1-p_{-})>0$
and $C=1$. For $\nu_{\rho}$-almost every $x$, the binary digit
frequency $p(x)\coloneqq\lim_{n\to\infty}\frac{1}{n}\sum^{n}_{j=1}\omega_{j}(x)$
exists, and the dyadic local dimension satisfies 
\begin{equation}
s(x)\coloneqq\lim_{n\to\infty}\frac{-\log\nu_{\rho}(Q_{n}(x))}{n}=h_{2}(p(x)).\label{eq:bernoulli-local-limit}
\end{equation}
Moreover, 
\begin{equation}
p_{*}\nu_{\rho}=\rho,\qquad s_{*}\nu_{\rho}=(h_{2})_{*}\rho.\label{eq:bernoulli-local-pushforward}
\end{equation}
Consequently, both dimension limits exist and 
\begin{equation}
h(\nu_{\rho})=\int^{p_{+}}_{p_{-}}h_{2}(p)\d\rho(p),\qquad D_{0}(\nu_{\rho})=\left(\int^{p_{+}}_{p_{-}}\frac{1}{h_{2}(p)}\d\rho(p)\right)^{-1}.\label{eq:bernoulli-two-dimensions}
\end{equation}

\end{proposition}
\begin{proof}
Since $p_{-}\leq p\leq p_{+}\leq1/2$, every factor in \eqref{eq:Bernoulli-cylinder}
is at most $1-p_{-}$. Integration proves \eqref{eq:bernoulli-mass-exponent}.
All dyadic intervals have positive mass, so the support is $[0,1]$.
The same uniform bound tends to zero with $n$ and proves atomlessness,
including at the dyadic endpoints.

Conditional on $P=p$, the strong law of large numbers gives $p(X)=p$
almost surely. Integrating this statement with respect to $\rho$
proves existence of the digit frequency and $p_{*}\nu_{\rho}=\rho$.
In particular, $p(x)\in\operatorname{supp}\rho$ for $\nu_{\rho}$-almost
every $x$.

It remains to compute the cylinder exponent of the mixture itself,
rather than that of the selected component. Let $x$ have digit frequencies
$p_{n}(x)\to p\in\operatorname{supp}\rho$, and put 
\[
f_{n}(v)\coloneqq p_{n}(x)\log v+(1-p_{n}(x))\log(1-v),\qquad v\in[p_{-},p_{+}].
\]
These functions converge uniformly on $[p_{-},p_{+}]$ to 
\[
f_{p}(v)\coloneqq p\log v+(1-p)\log(1-v).
\]
The function $f_{p}$ attains its maximum at $v=p$, because $f_{p}'(v)=\left(p-v\right)/\left(v(1-v)\ln2\right).$By
\eqref{eq:bernoulli-mixture-cylinder}, 
\[
\nu_{\rho}(Q_{n}(x))=\int^{p_{+}}_{p_{-}}2^{nf_{n}(v)}\d\rho(v).
\]
The upper bound by $2^{n\sup f_{n}}$ gives 
\[
\limsup_{n\to\infty}\frac{1}{n}\log\nu_{\rho}(Q_{n}(x))\leq f_{p}(p).
\]
 For the reverse bound, fix $\delta>0$. Continuity gives a relative
neighbourhood $U$ of $p$ in $[p_{-},p_{+}]$ on which $f_{p}(v)\geq f_{p}(p)-\delta$.
Since $p\in\operatorname{supp}\rho$, we have $\rho(U)>0$. Uniform
convergence therefore gives, for all sufficiently large $n$,
\[
\nu_{\rho}(Q_{n}(x))\geq\rho(U)2^{n(f_{p}(p)-2\delta)}.
\]
After taking logarithms, dividing by $n$, and letting first $n\to\infty$
and then $\delta\downarrow0$, we obtain 
\[
\lim_{n\to\infty}\frac{1}{n}\log\nu_{\rho}(Q_{n}(x))=p\log p+(1-p)\log(1-p).
\]
This proves \eqref{eq:bernoulli-local-limit} and the second identity
in \eqref{eq:bernoulli-local-pushforward}. No density or regularity
of $\rho$ was used. Finally, \eqref{eq:bernoulli-mass-exponent}
permits application of Remark~\ref{cor:local-harmonic}, which yields
\eqref{eq:bernoulli-two-dimensions}.
\end{proof}

\begin{remark}[Uniform parameters versus uniform dimensions] For
the uniform parameter law $d\rho(p)=dp/(p_{+}-p_{-})$ on $[p_{-},p_{+}]$
the local dimension has law $(h_{2})_{*}\rho$ with essential range
$[h_{2}(p_{-}),h_{2}(p_{+})]$, and the all-zero cylinder has maximal
mass $\max_{Q\in\D_{n}}\nu_{\rho}(Q)=\frac{(1-p_{-})^{n+1}-(1-p_{+})^{n+1}}{(p_{+}-p_{-})(n+1)}$,
so $\dim_{\infty}(\nu_{\rho})=-\log(1-p_{-})>0$. To prescribe the
dimension law itself, rather than the parameter law, change the mixing
weights as follows.\end{remark}

\subsection{Prescribing the distribution of the local dimension}

Let $\psi:[0,1]\to[0,1/2]$ denote the inverse of the continuous,
strictly increasing function $h_{2}|_{[0,1/2]}$.

\begin{corollary}[Arbitrary prescribed local-dimension law]\label{cor:prescribed-local-law}
Let $0<s_{-}<s_{+}\leq1$, and let $\eta$ be any Borel probability
measure supported on $[s_{-},s_{+}]$. Define 
\[
\nu_{\eta}\coloneqq\int^{s_{+}}_{s_{-}}\mu_{\psi(s)}\d\eta(s).
\]
Then $\nu_{\eta}$ is atomless, has full support $[0,1]$, and satisfies
\[
\nu_{\eta}(Q)\leq(1-\psi(s_{-}))^{n}\quad(Q\in\D_{n}).
\]
Its dyadic local dimension $s(x)$ exists $\nu_{\eta}$-almost everywhere
and has the prescribed law 
\[
s_{*}\nu_{\eta}=\eta.
\]
In particular, 
\[
h(\nu_{\eta})=\int^{s_{+}}_{s_{-}}s\d\eta(s),\qquad D_{0}(\nu_{\eta})=\left(\int^{s_{+}}_{s_{-}}s^{-1}\d\eta(s)\right)^{-1}.
\]

\end{corollary}
\begin{proof}
Apply Proposition~\ref{prop:bernoulli-local-law} to $\rho=\psi_{*}\eta$
on $[\psi(s_{-}),\psi(s_{+})]$. Since $h_{2}\circ\psi$ is the identity,
the law of the local dimension is $(h_{2})_{*}(\psi_{*}\eta)=\eta$.
The mass bound and the two dimension formulas follow from the same
proposition.
\end{proof}

\begin{remark}[The full feasible region]\label{rem:feasible-region}
More precisely, for any $0<s_{-}<s_{+}\leq1$, every pair $(H,D)$
satisfying 
\[
s_{-}\leq H\leq s_{+},\qquad\frac{s_{-}s_{+}}{s_{-}+s_{+}-H}\leq D\leq H
\]
is realised by a measure with almost-everywhere local dimension in
$[s_{-},s_{+}]$. For $s_{-}<H<s_{+}$, set $w=(s_{+}-H)/(s_{+}-s_{-})$
and take 
\[
\eta_{\vartheta}=(1-\vartheta)\delta_{H}+\vartheta\bigl(w\delta_{s_{-}}+(1-w)\delta_{s_{+}}\bigr),\qquad0\leq\vartheta\leq1.
\]
Its arithmetic mean is $H$, while its harmonic mean is the reciprocal
of $(1-\vartheta)/H+\vartheta(s_{-}+s_{+}-H)/(s_{-}s_{+})$, which
traverses the whole admissible interval. Apply Corollary~\ref{cor:prescribed-local-law}.
The endpoint cases use a point mass. These are bounds on the allowed
local values; when $D=H$, their essential range necessarily reduces
to the singleton $\{H\}$.

\end{remark}
\begin{example}[A prescribed uniform local-dimension law]
\label{ex:uniform-local-dimension}\label{ex:half-to-one-local-dimension}
For $0<s_{-}<s_{+}\leq1$, take 
\[
\nu=\frac{1}{s_{+}-s_{-}}\int^{s_{+}}_{s_{-}}\mu_{\psi(s)}\d s.
\]
Then the almost-everywhere local limit has law $s_{*}\nu=\lambda|_{[s_{-},s_{+}]}/(s_{+}-s_{-})$,
the normalised Lebesgue measure on $[s_{-},s_{+}]$, the support is
$[0,1]$, and $\dim_{\infty}(\nu)>0$. Equivalently, with $p_{-}=\psi(s_{-})$
and $p_{+}=\psi(s_{+})$, the substitution $s=h_{2}(p)$ gives 
\[
\nu=\frac{1}{s_{+}-s_{-}}\int^{p_{+}}_{p_{-}}\log\frac{1-p}{p}\,\mu_{p}\d p.
\]
Corollary~\ref{cor:prescribed-local-law} gives the arithmetic and
harmonic means 
\[
h(\nu)=\frac{s_{-}+s_{+}}{2},\qquad D_{0}(\nu)=\frac{s_{+}-s_{-}}{\ln s_{+}-\ln s_{-}}<h(\nu).
\]
The quantization value is the logarithmic mean of the endpoints. By
Remark~\ref{rem:heurteaux-dimensions}, the classical measure dimensions
instead record the endpoints: 
\[
\dim_{*}(\nu)=\mathrm{Dim}_{*}(\nu)=s_{-},\qquad\dim^{*}(\nu)=\mathrm{Dim}^{*}(\nu)=s_{+}.
\]
In particular, $s_{-}=1/2$ and $s_{+}=1$ give 
\[
h(\nu)=\frac{3}{4},\qquad D_{0}(\nu)=\frac{1}{2\ln2}=0.721347\ldots.
\]
Every non-empty relatively open subinterval of $[s_{-},s_{+}]$ receives
positive probability under $s$, but each individual level set has
$\nu$-mass zero. Thus almost-everywhere existence of the local dimension
need not imply exact-dimensionality.
\end{example}

\end{document}